\documentclass{article}
\usepackage{amsmath,systeme,geometry,mathrsfs,amssymb,dsfont}
\usepackage{enumitem}
\usepackage{amsthm}
\usepackage{graphicx}
\usepackage{float}
\usepackage{derivative}
\usepackage{amsfonts}

\usepackage[colorlinks,linkcolor=blue,anchorcolor=red,citecolor=blue]{hyperref}

\title{Long-Time Existence and Behavior of Solutions to the Inhomogeneous Kinetic FPU Equation}
\author{Haoling Xiang}
\date{}
\newtheorem{theorem}{Theorem}[section]
\newtheorem{proposition}[theorem]{Proposition}
\newtheorem{lemma}[theorem]{Lemma}
\newtheorem{corollary}[theorem]{Corollary}
\newtheorem{remark}[theorem]{Remark}
\theoremstyle{definition}

\begin{document}
\maketitle
\begin{abstract}
We study the inhomogeneous kinetic Fermi--Pasta--Ulam (FPU) equation, a nonlinear transport
equation describing the evolution of phonon density distributions with four-phonon
interactions.
The equation combines free transport in physical space with a nonlinear collision operator
acting in momentum space and exhibiting structural degeneracies.

We develop a functional framework that captures the interplay between spatial transport
and the degeneracies arising in the collision operator.
A key ingredient of the analysis is a dispersive estimate for the transport flow, which
quantifies decay effects generated by spatial propagation.

Using this dispersive mechanism, we obtain improved bounds for the nonlinear collision
operator and show that small solutions near the vacuum can be propagated on time scales
significantly longer than those dictated by conservation laws alone.
In particular, dispersion allows one to extend the classical quadratic lifespan to a
quartic time scale.
\end{abstract}

\section{Introduction}

\subsection{Background on kinetic FPU equations}

The Fermi--Pasta--Ulam (FPU) lattice model, originally introduced in the 1950s to investigate
the equipartition of energy among nonlinear oscillators~\cite{FPU1955}, has since become a
central paradigm for the study of nonlinear wave propagation, energy transport, and
thermalization phenomena in weakly anharmonic chains.

In the kinetic regime, the macroscopic evolution of energy-carrying excitations is described
in terms of phonon density distributions, where phonons are the quasiparticles associated
with lattice vibrations.
The resulting effective description takes the form of a Boltzmann-type kinetic equation,
commonly referred to as the \emph{kinetic FPU equation}.
Denoting by \(f=f(t,x,p)\) the phonon distribution function at time \(t\in\mathbb{R}_+\), with
spatial position \(x\in\mathbb{R}\) and periodic momentum variable
\(p\in\mathbb{T}=\mathbb{R}/2\pi\mathbb{Z}\), the inhomogeneous kinetic FPU equation reads
\begin{equation}\label{eq:kinetic_FPU}
\partial_t f + v(p)\,\partial_x f = \mathcal{C}[f],
\end{equation}
where \(v(p)=\omega'(p)\) denotes the group velocity associated with the dispersion relation
\(\omega(p)=|\sin(p/2)|\), and \(\mathcal{C}[f]\) is a collision operator encoding resonant
four-phonon interactions.
In the next subsection, we define the operator \(\mathcal{C}[f]\) explicitly and decompose it
into its gain and loss components.

\medskip
The \emph{homogeneous} kinetic FPU equation, where the distribution function depends only on
the momentum variable \(p\), has been extensively studied and provides a natural reference
point for the inhomogeneous problem. In this setting, the equation reduces to a phonon
Boltzmann equation describing resonant four-phonon interactions.

From a physical perspective, such phonon Boltzmann descriptions are expected to arise in the
weakly anharmonic (kinetic) regime; see Spohn~\cite{Spohn2006PhononBoltzmann} and
Aoki, Lukkarinen and Spohn~\cite{AokiLukkarinenSpohn2006}, and see also
Onorato--Lvov--Dematteis--Chibbaro~\cite{OnoratoLvovDematteisChibbaro2023}.
This viewpoint is rooted in the classical Boltzmann--Peierls theory of phonon transport~\cite{Peierls1929}.
A key mathematical motivation is provided by Lukkarinen and Spohn~\cite{LS08}, where anomalous
energy transport in the FPU--\(\beta\) chain is linked to the long-time behavior of a linearized
phonon Boltzmann equation; for physical reviews on anomalous conduction in low-dimensional
momentum-conserving lattices, see~\cite{LepriLiviPoliti2003,Dhar2008}.

More recently, Germain, La and Menegaki~\cite{GermainLaMenegaki2024RJ} analyzed the stability of
Rayleigh--Jeans equilibria for the kinetic FPU equation, establishing Lyapunov stability and
entropy production near equilibrium. Related work~\cite{EscobedoGermainLaMenegaki2025Entropy}
introduced an entropy maximization principle for kinetic wave equations on tori, identifying
Rayleigh--Jeans distributions as entropy maximizers under mass and energy constraints. Together,
these works provide a well-developed understanding of homogeneous kinetic dynamics near equilibrium.

\medskip
In contrast, the \emph{inhomogeneous} kinetic FPU equation couples spatial transport with
nonlinear resonant collisions. Spatially inhomogeneous phonon Boltzmann (or Peierls transport)
equations are standard kinetic models for phonon-mediated heat transport in weakly anharmonic
lattices; see, for instance, Spohn~\cite{Spohn2006PhononBoltzmann} and
Aoki, Lukkarinen and Spohn~\cite{AokiLukkarinenSpohn2006}. More generally, kinetic transport
equations can be derived from Hamiltonian dynamics in appropriate scaling limits; see
Spohn~\cite{Spohn1980MarkovianLimits}. A central question is whether, and to what extent, spatial transport can enhance stability and
relaxation by mitigating the nonlinear effects of resonant collisions in the kinetic FPU
setting.

\medskip
Related advances have been made in kinetic equations with transport, where dispersion of the
free flow is used to control nonlinear interactions.
Early work by Castella and Perthame~\cite{CastellaPerthame1996} initiated a dispersive approach
to kinetic transport by establishing Strichartz-type estimates for the free flow.
This dispersive viewpoint has subsequently been employed in kinetic settings where transport
interacts with nonlinear effects, notably in the study of the Boltzmann equation by
Ars\'enio~\cite{Arsenio2011}.
More recently, Ampatzoglou and L\'eger~\cite{AmpatzoglouLeger2024} applied closely related ideas
to an inhomogeneous kinetic wave equation, obtaining global existence and scattering for small
data.
Other techniques have been developed for kinetic models, including hypocoercive/entropy methods
(e.g.~\cite{DesvillettesVillani2001}), phase mixing (e.g.~\cite{MouhotVillani2011}),
and energy/semigroup (or spectral) approaches for Boltzmann-- and Landau-type equations
(e.g.~\cite{Guo2002,BriantGuo2016}). These works provide a complementary perspective to the present
dispersive approach. In our setting, the inhomogeneous kinetic FPU equation presents additional
difficulties due to the degeneracies of the dispersion relation and the resonance structure of
the collision operator.

\medskip
Our analysis builds on the parametrization of the resonant manifold introduced by
Lukkarinen and Spohn~\cite{LS08}, which underlies the recent stability results of
Germain, La and Menegaki~\cite{GermainLaMenegaki2024RJ} for the homogeneous phonon
Boltzmann equation.
In the inhomogeneous setting, we combine dispersive estimates for the transport flow
with refined bounds on the collision operator derived from this resonant geometry,
yielding an extended lifespan for small solutions near the vacuum.

\medskip
\subsection{The Inhomogeneous Kinetic FPU Equation}

The kinetic FPU equation describes the macroscopic evolution of phonon densities
in weakly anharmonic lattice systems.
In the kinetic regime, it takes the form of a transport equation in physical space
coupled with a nonlinear collision operator accounting for resonant four-phonon
interactions.

We consider the inhomogeneous kinetic FPU equation for the phonon distribution
function \(f=f(t,x,p)\),
\begin{equation}\label{eq:kinetic}
    \partial_t f(t,x,p) + \omega'(p)\,\nabla_x f(t,x,p)
    = \mathcal{C}[f](t,x,p),
\end{equation}
where \(x\in\mathbb{R}\), \(p\in\mathbb{T}=\mathbb{R}/2\pi\mathbb{Z}\), and 
\begin{equation}\label{def:omega}\omega(p)=|\sin(p/2)|,\end{equation} 
is the dispersion relation.

The collision operator \(\mathcal{C}\) is given by
\begin{equation}\label{eq:collision_general}
    \mathcal{C}[f](t,x,p_0)
    = \int_{\mathbb{T}^3}
        \delta(\Sigma)\,\delta(\Omega)
        \prod_{\ell=0}^3 \omega_\ell
        \prod_{\ell=0}^3 f_\ell
        \left( \frac{1}{f_0} + \frac{1}{f_1}
        - \frac{1}{f_2} - \frac{1}{f_3} \right)
        \, dp_1\,dp_2\,dp_3,
\end{equation}
where we use the shorthand
\[
p = p_0, \qquad
f = f_0, \qquad
\omega=\omega_0,\qquad
f_i = f(p_i), \qquad 
\omega_i=\omega(p_i),\quad i=0,1,2,3.
\]
The resonance functions are defined by
\[
\Sigma(p_0,p_1,p_2,p_3)=p_0+p_1-p_2-p_3, \qquad
\Omega(p_0,p_1,p_2,p_3)=\omega_0+\omega_1-\omega_2-\omega_3.
\]
As shown by Lukkarinen and Spohn~\cite{LS08}, the common zero set of the resonance
functions \(\Sigma\) and \(\Omega\) consists of trivial and nontrivial resonances.
Although the integrand vanishes on the trivial resonant set, the product
\(\delta(\Sigma)\delta(\Omega)\) is not a priori well defined in the sense of distributions.
In the linearized setting, Lukkarinen and Spohn~\cite{LS08} showed, via a suitable
regularization procedure, that the contribution of trivial resonances vanishes.

In the nonlinear setting, Germain, La and Menegaki~\cite{GermainLaMenegaki2024RJ}
take this fact for granted and develop a parametrization of the nontrivial resonant
manifold to analyze the collision operator.
Following these works, we disregard the trivial resonances and focus on the
nontrivial resonant interactions.

The collision operator admits a natural decomposition into gain and loss terms,
which will play a central role in the analysis.
A detailed discussion of its resonant structure and geometric properties is
deferred to Section~2.

\vspace{1em}
\subsection{Main results}

We now state simplified versions of our main results.
Precise assumptions, functional frameworks, and proofs are given in
Sections~3--5.

\begin{theorem}[Local well-posedness]\label{thm:intro_LWP}
For initial data bounded in suitable weighted \(L^\infty\) spaces in both space and
momentum, the inhomogeneous kinetic FPU equation admits a unique mild solution on a
time interval \(0 \le t \le T\), with lifespan
\[
T \sim \|f_0\|^{-2},
\]
where \(\|\cdot\|\) denotes the corresponding weighted \(L^\infty\) norm.
\end{theorem}

\medskip

Beyond this classical local theory, dispersive effects induced by the transport
flow allow one to work in finer mixed space--momentum norms and to propagate
smallness on longer time scales.

\begin{theorem}[Long-time existence near the vacuum]\label{thm:intro_longtime}
There exists \(\varepsilon_0>0\) such that for all \(0<\varepsilon\le\varepsilon_0\),
if the initial datum satisfies suitable smallness, integrability, and dispersive
bounds of size \(\varepsilon\), then the inhomogeneous kinetic FPU equation admits
a unique mild solution on a time interval \(0 \le t \le T\) with
\[
T \sim \varepsilon^{-4}.
\]
Here \(\varepsilon\) measures the size of the initial datum in the norms appearing
in the above assumptions (see Sections~5 for precise statements).
\end{theorem}
\medskip

The above results show that transport-induced dispersion plays a central role in
the inhomogeneous kinetic FPU equation.
At the linear level, we establish dispersive bounds for the free transport flow
in weighted spaces adapted to the degeneracy of the FPU dispersion relation.
At the nonlinear level, we derive refined estimates for the gain and loss terms
of the collision operator based on the geometry of the resonant manifold.
Combining these two ingredients yields a dispersive stabilization mechanism that
upgrades the classical quadratic lifespan to a quartic time scale.

\vspace{1em}
\subsection{Outline of the Paper}
This paper is organized as follows.

\begin{enumerate}[leftmargin=2em]

\item \textbf{Section~2: Gain--loss decomposition and resonant parametrization.}
We recall the gain--loss decomposition of the collision operator and the
parametrization of the nontrivial resonant manifold introduced in~\cite{LS08}
and used in~\cite{GermainLaMenegaki2024RJ}.
These tools allow one to rewrite the collision operator in a form suitable
for the weighted estimates developed later.

\item \textbf{Section~3: Dispersive estimates and short-time well-posedness.}
We establish weighted dispersive estimates for the free transport semigroup
associated with~\eqref{eq:kinetic}.
These linear bounds are formulated in mixed space--momentum norms adapted to
the degeneracy of the group velocity and yield a short-time well-posedness
theory for mild solutions.

\item \textbf{Section~4: A priori bounds for the collision operator.}
Section~4 is devoted to the nonlinear analysis of \(\mathcal{C}\).
Using the resonant parametrization, we prove weighted a priori bounds for all
gain and loss components.
The main step is the control of the principal gain term via geometric properties
of the resonant manifold, ensuring compatibility with the dispersive framework
introduced in Section~5.

\item \textbf{Section~5: Long-time existence via dispersive propagation.}
We combine the dispersive decay of the transport flow with the collision bounds
from Section~4 to propagate smallness on long time intervals.
We introduce a dispersive control norm and implement a bootstrap argument that
upgrades the classical quadratic lifespan to a quartic time scale, yielding the
main long-time well-posedness result for small data.
We also show that the resulting mild solutions preserve the physical invariants
of mass and energy.

\end{enumerate}

\medskip
\paragraph{Difficulties of the proof.}
We briefly outline the main analytical difficulties addressed in this work.
The inhomogeneous kinetic FPU equation is studied here in one spatial dimension and
near the vacuum, which constitutes the most challenging regime from a dispersive
viewpoint due to the extremely weak decay of the free transport flow.
In this setting, neither entropy methods nor spectral gap arguments are available.

Moreover, the dispersion relation $\omega(p)=|\sin(p/2)|$ is degenerate at low
frequencies, and the associated resonant manifold has a nontrivial geometry and low
codimension.
As a consequence, the collision operator exhibits strong singularities, making its
control delicate even at the perturbative level.

From a mathematical viewpoint, controlling long-time dynamics in this regime requires
combining dispersive decay generated by the transport flow with refined geometric
information on resonant interactions.
This difficulty is already present for classical collisional kinetic equations such as
the Boltzmann equation, where long-time well-posedness near vacuum is known only under
restrictive assumptions.
The present work addresses these obstacles by exploiting the interplay between weak
dispersive decay and the precise geometry of the FPU resonance manifold.

\medskip

\noindent\textbf{Notational conventions.}
We set $\langle t\rangle := (1+t^2)^{1/2}$.
Throughout the paper, $C$ denotes a positive constant whose value may change from
line to line.
We write $a\lesssim b$ to indicate that $a\le C\,b$ for such a constant $C$.
Unless otherwise specified, all $L_p^q$ norms are taken over the torus $\mathbb{T}$.

We use the standard mixed Lebesgue norms
\[
\|f\|_{L_x^rL_p^q}
:= \left( \int_{\mathbb{R}}
   \Big( \int_{\mathbb{T}} |f(x,p)|^q\,dp \Big)^{r/q} dx \right)^{1/r},
\qquad 1\le r,q\le\infty,
\]
with the usual modification when $r=\infty$ or $q=\infty$.
\vspace{1em}

\section{Gain--loss decomposition and resonant parameterization}

\subsection{Gain--loss decomposition}

In the spirit of the classical Boltzmann equation, the collision operator can be
decomposed into gain and loss contributions as
\begin{equation}\label{eq:C_decomp}
    \mathcal{C}[f]
    = \mathcal{G}_1[f,f,f] + \mathcal{G}_2[f,f,f]
      - \mathcal{L}_1[f,f,f] - \mathcal{L}_2[f,f,f].
\end{equation}
Each term is a trilinear form defined by
\begin{align}
    \mathcal{G}_1[f,g,k]
    &= \int_{\mathbb{T}^3}
        \delta(\Sigma)\,\delta(\Omega)
        \prod_{\ell=0}^3 \omega_\ell\,
        f_1 g_2 k_3
        \, dp_1\,dp_2\,dp_3, \label{eq:G1}\\
    \mathcal{G}_2[f,g,k]
    &= \int_{\mathbb{T}^3}
        \delta(\Sigma)\,\delta(\Omega)
        \prod_{\ell=0}^3 \omega_\ell\,
        f_0 g_2 k_3
        \, dp_1\,dp_2\,dp_3, \label{eq:G2}\\
    \mathcal{L}_1[f,g,k]
    &= \int_{\mathbb{T}^3}
        \delta(\Sigma)\,\delta(\Omega)
        \prod_{\ell=0}^3 \omega_\ell\,
        f_0 g_1 k_3
        \, dp_1\,dp_2\,dp_3, \label{eq:L1}\\
    \mathcal{L}_2[f,g,k]
    &= \int_{\mathbb{T}^3}
        \delta(\Sigma)\,\delta(\Omega)
        \prod_{\ell=0}^3 \omega_\ell\,
        f_0 g_1 k_2
        \, dp_1\,dp_2\,dp_3. \label{eq:L2}
\end{align}
We also introduce the multilinear operator
\begin{equation}\label{def:Cfgk}
    \mathcal{C}[f,g,k]
    := \mathcal{G}_1[f,g,k] + \mathcal{G}_2[f,g,k]
       - \mathcal{L}_1[f,g,k] - \mathcal{L}_2[f,g,k],
\end{equation}
so that \(\mathcal{C}[f]=\mathcal{C}[f,f,f]\).
This decomposition will be used throughout the paper to derive a priori
estimates and to isolate the distinct analytic behaviors of gain and loss terms.

\subsection{Parameterization of the resonant manifold}

The resonance conditions \(\Sigma=0\) and \(\Omega=0\) admit both trivial and
nontrivial solutions.
In the linearized setting, Lukkarinen and Spohn~\cite{LS08} showed, via a suitable
regularization procedure, that the contribution of trivial resonances
\(\{p_0,p_1\}=\{p_2,p_3\}\) vanishes.
In the nonlinear theory, this fact is taken for granted
(see~\cite{GermainLaMenegaki2024RJ}), and in what follows we restrict attention to
the nontrivial resonant manifold.

Following the parameterization introduced by Lukkarinen and Spohn~\cite{LS08} and
subsequently employed in~\cite{GermainLaMenegaki2024RJ}, the nontrivial resonant
manifold can be parameterized as follows.
Let \(p_0,p_2\in\mathbb{T}\), and choose representatives in \([0,2\pi]\).
Then the resonance relations are solved by
\begin{equation}\label{eq:param}
    p_1 = h(p_0,p_2) \mod 2\pi,
\end{equation}
where
\begin{equation}\label{eq:h_def}
    h(x,z)
    = \frac{z-x}{2}
    + 2\arcsin\!\left(
        \tan\!\big(\tfrac{|z-x|}{4}\big)
        \cos\!\big(\tfrac{z+x}{4}\big)
    \right).
\end{equation}
With this parameterization, the collision operator admits a one--dimensional
representation: for \(p_0\in[0,2\pi]\),
\begin{equation}\label{eq:collision_param}
    \mathcal{C}[f](p_0)
    = \int_0^{2\pi}
        \frac{\omega_0\omega_1\omega_2\omega_3}
             {\sqrt{F_+(p_0,p_2)}}
        \prod_{\ell=0}^3 f_\ell
        \left(
            \frac{1}{f}
            + \frac{1}{f_1}
            - \frac{1}{f_2}
            - \frac{1}{f_3}
        \right)
        \, dp_2,
\end{equation}
where
\[
    p_1 = h(p_0,p_2),
    \qquad
    p_3 = p_0 + p_1 - p_2,
\]
and
\begin{equation}\label{eq:Fplus}
    F_+(p_0,p_2)
    = \sqrt{
        \left[\cos\!\left(\frac{p_0}{2}\right)
        + \cos\!\left(\frac{p_2}{2}\right)\right]^2
        + 4 \sin\!\left(\frac{p_0}{2}\right)
          \sin\!\left(\frac{p_2}{2}\right)
    }.
\end{equation}

In particular, the gain and loss operators defined in
\eqref{eq:G1}--\eqref{eq:L2} can be rewritten, for \(p_0\in[0,2\pi]\), as
\begin{align}
    \mathcal{G}_1[f,g,k](p_0)
    &= \int_0^{2\pi}
        \frac{\omega_0\omega_1\omega_2\omega_3}
             {\sqrt{F_+(p_0,p_2)}}
        f_1 g_2 k_3 \, dp_2, \label{eq:G1_param}\\
    \mathcal{G}_2[f,g,k](p_0)
    &= \int_0^{2\pi}
        \frac{\omega_0\omega_1\omega_2\omega_3}
             {\sqrt{F_+(p_0,p_2)}}
        f_0 g_2 k_3 \, dp_2, \label{eq:G2_param}\\
    \mathcal{L}_1[f,g,k](p_0)
    &= \int_0^{2\pi}
        \frac{\omega_0\omega_1\omega_2\omega_3}
             {\sqrt{F_+(p_0,p_2)}}
        f_0 g_1 k_3 \, dp_2, \label{eq:L1_param}\\
    \mathcal{L}_2[f,g,k](p_0)
    &= \int_0^{2\pi}
        \frac{\omega_0\omega_1\omega_2\omega_3}
             {\sqrt{F_+(p_0,p_2)}}
        f_0 g_1 k_2 \, dp_2. \label{eq:L2_param}
\end{align}

These formulas will serve as the starting point for the weighted estimates
of the gain and loss operators developed in later sections.

\section{Dispersive Estimates and Local Well-posedness in Weighted \texorpdfstring{$L^\infty_{x,p}$}{L-infinity(x,p)} Spaces}
In this section, we investigate the dispersive behavior of the linear transport semigroup
\[
\mathcal{S}(t)f(x,p) = f(x - v(p)t, p),
\]
associated with the kinetic FPU equation.
We first establish a weighted mixed-norm dispersive estimate (Lemma~\ref{lem:SW-weighted}),
which quantifies the decay of the transport flow in the spatial variable
in terms of the momentum derivative $v'(p)$.
This estimate serves as a key analytic tool in controlling the nonlinear collision operator
and forms the foundation of the local well-posedness theory developed later in this section.
\medskip
We then apply this dispersive bound to construct mild solutions in weighted
\(L^\infty_{x,p}\) spaces.
Finally, we discuss several structural properties of the solution space,
including the commutation of the semigroup with the weight function and
embedding relations between weighted \(L^q\) spaces.
Together, these results provide a robust linear and functional framework
for the subsequent analysis of the nonlinear kinetic equation.
\\
We define the weighted \(L^q_p\) norm by
\begin{equation}\label{eq:weighted_Lq}
    \|f\|_{L_p^q(W)}
    = \left( \int_{\mathbb{T}} |f(p)\,W(p)|^q \, dp \right)^{1/q},
    \qquad 0 < q < \infty.
\end{equation}
In the case \(q = \infty\), the corresponding norm is given by
\begin{equation}\label{eq:weighted_Linf}
    \|f\|_{L_p^\infty(W)}
    = \sup_{p \in \mathbb{T}} |f(p)|\,W(p).
\end{equation}

\subsection{Dispersive Estimate for the Linear Semi-group}
We state here the weighted Stein--Weiss interpolation theorem (see ~\cite{SteinWeiss1958,Gustavsson1982}).
\begin{lemma}[Weighted Stein--Weiss Interpolation]\label{lem:stein-weiss}
Let $(X,\mu)$ and $(Y,\nu)$ be measure spaces.  
For $j=0,1$, let $0 < p_j, q_j \le \infty$ and let $W_j : X \to [0,\infty)$,  
$V_j : Y \to [0,\infty)$ be measurable weight functions.  
Assume a linear operator $T$ satisfies
\begin{equation}\label{eq:SW-endpoint-bounds}
    \|T f\|_{L^{q_j}(V_j)}
    \;\le\; A_j \, \|f\|_{L^{p_j}(W_j)},
    \qquad j = 0,1,
\end{equation}
for all $f$ in a dense subspace of $L^{p_j}(W_j)$.

For any $0<\theta<1$, define the interpolated exponents
\[
    \frac{1}{p_\theta}
    = \frac{1-\theta}{p_0} + \frac{\theta}{p_1},
    \qquad
    \frac{1}{q_\theta}
    = \frac{1-\theta}{q_0} + \frac{\theta}{q_1},
\]
and the interpolated weights
\[
    W_\theta = W_0^{\,1-\theta} W_1^{\,\theta},
    \qquad
    V_\theta = V_0^{\,1-\theta} V_1^{\,\theta}.
\]

Then $T$ extends to a bounded map
\[
    T : L^{p_\theta}(W_\theta) \longrightarrow L^{q_\theta}(V_\theta),
\]
and satisfies the estimate
\begin{equation}\label{eq:SW-interpolated-bound}
    \|T f\|_{L^{q_\theta}(V_\theta)}
    \;\le\;
    A_0^{\,1-\theta} \, A_1^{\,\theta}
    \,\|f\|_{L^{p_\theta}(W_\theta)}.
\end{equation}

In particular,
\[
   (L^{p_0}(W_0),\, L^{p_1}(W_1))_{\theta}
   = L^{p_\theta}(W_\theta)
   \quad \text{with equivalence of norms}.
\]
\end{lemma}

\begin{remark}
For later use in the proof of the dispersive estimates, we note that
Lemma~\ref{lem:stein-weiss} extends verbatim to mixed Lebesgue
spaces of the form $L_x^r L_p^q(W)$.
Indeed, mixed Lebesgue spaces form an interpolation scale, and
multiplicative weights interpolate by the geometric mean.
Consequently,
\[
   \big( L_x^{r_0} L_p^{q_0}(W_0),\;
         L_x^{r_1} L_p^{q_1}(W_1) \big)_\theta
   = L_x^{r_\theta} L_p^{q_\theta}(W_\theta),
   \qquad
   W_\theta = W_0^{1-\theta} W_1^{\theta},
\]
with equivalence of norms.
No additional assumptions are required in the mixed setting.
\end{remark}

\begin{lemma}[Weighted dispersive estimate for mixed norms]\label{lem:SW-weighted}
The linear semi-group is defined as \begin{equation}\label{def:S}\mathcal{S}(t)f(x,p)=f(x-v(p)t,p),\end{equation}  with $v(p)=\omega'(p),$ 
and
$\omega(p)=|\sin(p/2)|$. 
\\
Fix $1\le q\le r\le\infty$. Then for all $t\in \mathbb{R}$,
\begin{equation}\label{eq:SW-mixed}
  \|\mathcal{S}(t)f\|_{L_x^{r}L_p^{q}}
  \;\le\; \Big|\frac{t}{4}\Big|^{-\big(\frac1q-\frac1r\big)}\;
  \|f\|_{L_x^{q}\,L_p^{r}\!\big(\big(\omega(p)\big)^{-\left(\frac1q-\frac1r\right)}\big)}.
\end{equation}
Moreover, for $1\le q =r\le \infty$, we have
\[
  \|\mathcal{S}(t)f\|_{L_x^{r}L_p^{q}}
  \;=\;
  \|f\|_{L_x^{q}L_p^{r}}.
\]
\end{lemma}

\begin{proof}
View $\mathcal{S}(t)$ as a linear operator in the $x$-variable with vector-valued
range/target spaces $L_p^{q}$.

\emph{Endpoint 1 (conservation):}
The
translation invariance gives
\[
  \|\mathcal{S}(t)f\|_{L_x^{1}L_p^{1}}=\|f\|_{L_x^{1}L_p^{1}}.
\]

\emph{Endpoint 2 (change of variables):}
By the 1D change of variables $y=x-v(p)t$, 
\[
  \|\mathcal{S}(t)f\|_{L_x^\infty L_p^{1}}
  \;\le\; \Big|\frac{t}{4}\Big|^{-1}\,\|f\|_{L_x^{1}L_p^{\infty}\!\big(\big(\omega(p)\big)^{-1}\big)}.
\]
Identifying $\mathbb T$ with $[0,2\pi]$ (endpoints identified), we note that
$v(p)=\tfrac12\cos(p/2)$ is strictly decreasing on $[0,2\pi]$; hence for $t\neq0$
the change of variables $y=x-v(p)t$ is valid (up to a null set).

Now apply the weighted Stein--Weiss interpolation theorem (See Lemma~\ref{lem:stein-weiss}):
interpolate between
\[
  \mathcal{S}(t):\ L_x^{1}L_p^{1} \to L_x^{1}L_p^{1}
  \quad\text{and}\quad
  \mathcal{S}(t):\ L_x^{1}L_p^{\infty}\big(\big(\omega(p)\big)^{-1}\big) \to L_x^{\infty}L_p^{1}
\]
One gets
\begin{equation}\label{q=1}
  \|\mathcal{S}(t)f\|_{L_x^{r_1}L_p^{1}}
  \;\le\; \Big|\frac{t}{4}\Big|^{-(1-\frac1{r_1})}\,\|f\|_{L_x^{1}L_p^{r_1}\!\big(\big(\omega(p)\big)^{-(1-\frac1{r_1})}\big)},\quad 1\le r_1\le \infty
\end{equation}
\emph{Endpoint 3 (conservation):}
The
translation invariance gives
\[
  \|\mathcal{S}(t)f\|_{L_x^{\infty}L_p^{\infty}}=\|f\|_{L_x^{\infty}L_p^{\infty}}.
\]
Now apply the weighted Stein--Weiss interpolation theorem (See Lemma~\ref{lem:stein-weiss}) between Endpoint 2 and Endpoint 3, i.e.
interpolate between
\[\mathcal{S}(t):\ L_x^{1}L_p^{\infty}\big(\big(\omega(p)\big)^{-1}\big) \to L_x^{\infty}L_p^{1}
  \quad\text{and}\quad
\mathcal{S}(t):\ L_x^{\infty}L_p^{\infty} \to L_x^{\infty}L_p^{\infty}
\]
One gets
\begin{equation}\label{r=infty}
  \|\mathcal{S}(t)f\|_{L_x^{\infty} L_p^{q_1}}
  \;\le\; \Big|\frac{t}{4}\Big|^{-\frac{1}{q_1}}\,\|f\|_{L_x^{q_1}L_p^{\infty}\big(\big( \omega(p) \big)^{-\frac{1}{q_1}}\big)},\quad 1\le q_1\le \infty
\end{equation}
Now apply the weighted Stein--Weiss interpolation theorem (See Lemma~\ref{lem:stein-weiss}) between \eqref{q=1} and \eqref{r=infty}.
One gets \eqref{eq:SW-mixed}.
\end{proof}

\begin{corollary}\label{cor:r=2 q=1-weighted}
Taking $(r,q)=(2,1)$ and $(r,q)=(\infty,2)$ in \eqref{eq:SW-mixed} gives
\[
  \|\mathcal{S}(t)f\|_{L_x^{2}L_p^{1}}
  \;\le\; \Big|\frac{t}{4}\Big|^{-1/2}\,
  \|f\|_{L_x^{1}L_p^{2}\!\big(\omega^{-1/2}\big)}.
\]
\[
  \|\mathcal{S}(t)f\|_{L_x^{\infty}L_p^{2}}
  \;\le\; \Big|\frac{t}{4}\Big|^{-1/2}\,
  \|f\|_{L_x^{2}L_p^{\infty}\!\big(\omega^{-1/2}\big)}.
\]

\end{corollary}
\begin{remark}
One way to eliminate the weight is to exploit the fact that $p$ is
restricted to a bounded domain. In particular, we apply the embedding
theory of $L^q$ spaces on bounded sets, choosing a larger exponent $q_+$
to absorb the weight. This leads to the following dyadic layer-set
dispersive estimate (in an unweighted form)
\begin{equation}\label{eq:dyadic-q-to-q}
  \|\mathcal{S}(t)f\|_{L_x^{q}L_p^{1}}
  \;\lesssim\; t^{-(1-1/q)}\,\|f\|_{L_x^{1}L_p^{q_+}},
  \qquad t>0,\ 1\le q<\infty.
\end{equation}

\begin{proof}
Partition $\mathbb T$ into dyadic level sets of $|v'(p)|$:
\[
  E_\alpha = \{\,p : \alpha/2 < |v'(p)| \le \alpha \,\}, \quad \alpha = 2^{-k},
\]
and
\[
  E_{>1} = \{\,p : |v'(p)| > 1 \,\}.
\]
Since $v'(p)=-\tfrac14\sin(p/2)$ has only simple zeros, we have
\[
  |E_\alpha|\lesssim \alpha.
\]
Take $g=|f|$ and define
\[
  T_t^{E}g(x) := \int_E g(x-v(p)t,p)\,dp.
\]
The one-dimensional change of variables $y=x-v(p)t$ gives
\[
\int_{E_\alpha} |g(x-v(p)t,p)|\,dp
= \int_{\{y:\,\alpha/2<\frac14\sqrt{1-4({\frac{x-y}{t}})^2}\le \alpha\}}
   \frac{|g(y,p_x(y))|}{t|v'(p_x(y))|}
   \, dy
\lesssim (t\alpha)^{-1}\int_{\mathbb{R}}\int_{E_\alpha}|g(y,p)| \,dp\,dy.
\]
Taking the $L_x^\infty$ norm gives
\[
  \|T_t^{E_\alpha}g\|_{L_x^\infty}
  \lesssim (t\alpha)^{-1}\|g\|_{L_x^1L_p^1(E_\alpha)}.
\]
Moreover, from the definition of $T_t^{E_\alpha}g$, we directly get
\[
  \|T_t^{E_\alpha}g\|_{L_x^1}
  \le \|g\|_{L_x^1L_p^1(E_\alpha)}.
\]
From the above two inequalities, interpolating between $L_x^1$ and
$L_x^\infty$ gives
\[
  \|T_t^{E_\alpha}g\|_{L_x^q}
   \lesssim (t\alpha)^{-(1-1/{q})}\|g\|_{L_x^1L_p^1(E_\alpha)}.
\]
By Hölder’s inequality in $p$, it holds that
\[
  \|g\|_{L_x^1L_p^1(E_\alpha)}
  \le |E_\alpha|^{1-1/{q_+}}\|g\|_{L_x^1L_p^{q_+}(E_\alpha)}.
\]
Thus,
\[
  \|T_t^{E_\alpha}g\|_{L_x^q}
  \lesssim \alpha^{1/q-1/{q_{+}}} t^{-(1-1/q)}
           \|g\|_{L_x^1L_p^{q_+}(E_\alpha)}
  \lesssim \alpha^{1/q-1/{q_{+}}} t^{-(1-1/q)}
           \|g\|_{L_x^1L_p^{q_+}}.
\]
Summing in $\alpha$ and using Minkowski's inequality yields
\eqref{eq:dyadic-q-to-q}.
\end{proof}
The dyadic layer-set argument requires an exponent $q_+>q$.  
Indeed, on each dyadic set $E_\alpha=\{p:\alpha/2<|v'(p)|\le\alpha\}$ one obtains
\[
\|T_t^{E_\alpha}g\|_{L_x^q}
   \lesssim t^{-(1-1/q)}\,\alpha^{\,1/q-1/q_+}\,
   \|g\|_{L_x^1L_p^{q_+}(E_\alpha)},
\]
and since $|E_\alpha|\lesssim\alpha$, the sum in $\alpha$ converges only when
$1/q-1/q_+>0$.  
Thus the unweighted dyadic method cannot yield
\[
\mathcal{S}(t): L_x^1L_p^{q}\to L_x^qL_p^{1},\qquad q>1.
\]

To compensate for the small-$\alpha$ degeneracy, one may either  
(i) introduce a weight in $p$ and apply weighted interpolation, or  
(ii) localize away from $\{\omega=0\}$.  
These remedies restore the dispersive decay with the expected rate.
\end{remark}

\vspace{1em}

\subsection{Structure of the Weighted Function Spaces}
Recall the weight function $\omega(p)$ defined in \eqref{def:omega}. We now study a basic property of the weighted $L^q$ spaces associated with this weight function.

\begin{lemma}[Basic properties of the weight and weighted embeddings]\label{lem:weight-properties}
Let $\mathcal{S}(t)$ be the operator defined in \eqref{def:S}. Then the following hold:

\smallskip
\noindent\textbf{(i) Commutation with $\mathcal{S}(t)$.}
For every $l\in\mathbb{R}$,
\begin{equation}\label{eq:commute}
    \mathcal{S}(t)\big(\omega^l(p)f\big)=\omega^l(p)\,\mathcal{S}(t)f .
\end{equation}

\smallskip
\noindent\textbf{(ii) Monotonicity in $L^q_p$.}
For any $1\le q\le\infty$ and $l_2\le l_1$,
\begin{equation}\label{eq:weight-monotonicity}
    \|\omega^{l_1}(p) f\|_{L_p^q}
    \;\le\;
    \|\omega^{l_2}(p) f\|_{L_p^q}.
\end{equation}

\smallskip
\noindent\textbf{(iii) Weighted embedding between different $L^q$ spaces.}
For any exponents
\[
\infty \ge q_1 > q_2 \ge 1,\qquad 
m < \frac{1}{q_2}-\frac{1}{q_1},
\]
and any $l\in\mathbb{R}$, we have
\begin{equation}\label{eq:weighted_embedding}
    \|\omega^{l}(p) f\|_{L_p^{q_2}}
    \;\lesssim\;
    \|\omega^{\,l+m}(p) f\|_{L_p^{q_1}},
\end{equation}
with a constant depending only on $(q_1,q_2,m)$.
\end{lemma}

\begin{proof}
\textbf{(i) Commutation.}  
Since $\mathcal{S}(t)$ acts by translation in $x$ and leaves the
$p$--variable unchanged,
\[
\mathcal{S}(t)(\omega^l f)(x,p)=\omega^l(p)\,f(x-v(p)t,p)
=\omega^l(p)\,(\mathcal{S}(t)f)(x,p).
\]

\smallskip
\textbf{(ii) Monotonicity.}  
From $0\le\omega\le1$ and $l_2\le l_1$ we have
$\omega^{l_1}(p)\le\omega^{l_2}(p)$ pointwise. 
Taking the $L_p^q$-norm gives \eqref{eq:weight-monotonicity}.

\smallskip
\textbf{(iii) Weighted embedding.}
Let $m<\frac{1}{q_2}-\frac{1}{q_1}$.
Write
\[
\omega^{l}f
= \omega^{-m} \cdot \omega^{l+m} f.
\]
Applying Hölder with exponents
\[
\frac{1}{q_2}=\frac{1}{q_1}+\frac{1}{r},
\qquad 
r=\frac{q_1 q_2}{q_1-q_2},
\]
gives
\[
\|\omega^{l} f\|_{L_p^{q_2}}
\le
\|\omega^{-m}\|_{L_p^{r}}
\,
\|\omega^{l+m} f\|_{L_p^{q_1}}.
\]

We have $\|\omega^{-m}\|_{L_p^{r}}<C(q_1,q_2,m)<\infty$ since \(mr<1\), which is exactly our assumption. Therefore
\[
\|\omega^{l}(p) f\|_{L_p^{q_2}}
\;\lesssim\;
\|\omega^{l+m}(p) f\|_{L_p^{q_1}},
\]
which proves \eqref{eq:weighted_embedding}.
\end{proof}
\vspace{1em}
\subsection{Local Well-Posedness in Weighted \texorpdfstring{$L^\infty_{x,p}$}{L-infinity(x,p)} Spaces}
Define the mild solution of the inhomogeneous FPU \eqref{eq:kinetic} by Duhamel formula
\begin{equation}\label{Duhamel}
    f(t)=\mathcal{S}(t)f_0+\int_0^t \mathcal{S}(t-s)\mathcal{C}[f](s)\, ds.
\end{equation}

\begin{theorem}[Local Well-posedness]\label{Theorem:LWP_L_infinity}
Let $f_0 \in L_{x,p}^\infty(\omega^{\alpha})$ with $\alpha \leq 1$.  
Then the inhomogeneous FPU \eqref{eq:kinetic} admits a unique mild solution 
\[
    f \in C\!\left([0,T],\, L_x^\infty L_p^\infty(\omega^{\alpha})\right),
\]
and the time of existence satisfies
\begin{align}
    T \sim \|f_0\|_{L_{x,p}^\infty(\omega^{\alpha})}^{-2}.
\end{align}
\end{theorem}

\begin{proof}
The proof follows from the conservation law and the commutation property of the weights with $\mathcal{S}(t)$~\eqref{eq:commute} 
\begin{equation}\label{LinftyS}
  \|\mathcal{S}(t)f\|_{L_x^{\infty}L_p^{\infty}(\omega^{\alpha})} = \|f\|_{L_x^{\infty}L_p^{\infty}(\omega^{\alpha})},
\end{equation}
together with the boundedness of the collision operator in 
$L_p^\infty(\omega^{\alpha})$ (see \cite{GermainLaMenegaki2024RJ}, Sec.~3.1).  
More precisely, for $\alpha \le 1$ and $\mathcal{C}[f,g,k]$ defined in \eqref{def:Cfgk}, we have
\begin{equation}\label{LinftyC}
    \|\mathcal{C}[f,g,k]\|_{L_p^\infty(\omega^{\alpha})}
    \;\lesssim\;
    \|f\|_{L_p^\infty(\omega^{\alpha})}  \|g\|_{L_p^\infty(\omega^{\alpha})}  \|k\|_{L_p^\infty(\omega^{\alpha})}.
\end{equation}
Define the mapping as 
\[\Phi[f](t):=\mathcal{S}(t)f_0+\int_0^t \mathcal{S}(t-s)\mathcal{C}[f](s)\, ds.\]
It follows from \eqref{LinftyS} and \eqref{LinftyC} that
\[
\sup_{\{0\leq t\leq T\}}\|\Phi[f]\|_{L_{x,p}^\infty(\omega^{\alpha})}\lesssim \|f_0\|_{L_{x,p}^\infty(\omega^{\alpha})}+T\sup_{\{0\leq t\leq T\}}\|f(t)\|^3_{L_{x,p}^\infty(\omega^{\alpha})},
\]
which, together with the Banach fixed point theorem, yields the result.
\end{proof}
\vspace{1em}
\section{A Priori Estimates for the Collision Operator}

In this section we establish weighted estimates for the trilinear components 
of the collision operator $\mathcal{C}[f]$ defined in \eqref{eq:collision_general}.  
The decomposition introduced in \eqref{eq:C_decomp} allows us to treat separately 
the gain and loss contributions $\mathcal{G}_i$ and $\mathcal{L}_i$ $(i=1,2)$, 
each of which exhibits slightly different analytical behavior.

The estimates for $\mathcal{L}_2$, $\mathcal{L}_1$, and $\mathcal{G}_2$ 
follow in a relatively straightforward manner from the parametrization 
of the resonance manifold and the lower bound on $\sqrt{F_+}$.
These bounds rely on elementary symmetry properties of 
$\Sigma$ and $\Omega$ with respect to the interchange of $(p_2,p_3)$, 
and yield natural weighted inequalities involving mixed $L^q_p$ norms.

The main technical difficulty arises in the estimate for $\mathcal{G}_1$, 
where the natural cancellations present in the loss terms are no longer available, 
and the structure of the integral kernel requires a more delicate analysis.  
In particular, controlling the singularity of the Jacobian near resonant points 
demands a refined use of the weight $\omega(p)=|\sin(p/2)|$ and 
a careful decomposition of the integration domain.

We next establish each of these bounds in detail.

\begin{lemma}\label{lem:L2_bound}
For any \(1 \le q \le \infty\) and \(l \in \mathbb{R}\), the trilinear form 
\(\mathcal{L}_2[f,g,k]\) satisfies the estimate
\begin{equation}\label{bound_L2}
    \|\omega^l \mathcal{L}_2[f,g,k]\|_{L_p^q}
    \lesssim
    \|\omega^{l+\frac{1}{2}} f\|_{L_p^q}\,
    \|\omega g\|_{L_p^\infty}\,
    \|\omega^{\frac{1}{2}} k\|_{L_p^1}.
\end{equation}
\end{lemma}

\begin{proof}
Recall that
\[
F_+(p_0,p_2)
    = \left[\cos\!\left(\tfrac{p_0}{2}\right)
        + \cos\!\left(\tfrac{p_2}{2}\right)\right]^2
      + 4 \sin\!\left(\tfrac{p_0}{2}\right)
        \sin\!\left(\tfrac{p_2}{2}\right),
\]
so that
\begin{equation}\label{F+lower}
\sqrt{F_+(p_0,p_2)}
    \ge 2 \sqrt{\sin\!\left(\tfrac{p_0}{2}\right)
          \sin\!\left(\tfrac{p_2}{2}\right)}
    = 2 \,\omega_0^{1/2}\omega_2^{1/2}.
\end{equation}
Using this lower bound, we have
\begin{align*}
\|\omega^l \mathcal{L}_2[f,g,k]\|_{L_p^q}
&= \Bigg\|\omega_0^l
    \int_0^{2\pi}
    \frac{\omega_0\omega_1\omega_2\omega_3}
         {\sqrt{F_+(p_0,p_2)}}\,
    f(p_0)\,g(p_1)\,k(p_2)\,
    dp_2
  \Bigg\|_{L_p^q} \\[4pt]
&\lesssim
  \Bigg\|
    \omega_0^{\,l+\frac{1}{2}}
    \int_0^{2\pi}
      f(p_0)\,g(p_1)\,k(p_2)\,
      \omega_1\,\omega_2^{1/2}\,\omega_3\,
    dp_2
  \Bigg\|_{L_p^q} \\[4pt]
&\le
  \|\omega^{l+\frac{1}{2}} f\|_{L_p^q}
  \Bigg\|
    \int_0^{2\pi}
      g(p_1)\,k(p_2)\,
      \omega_1\,\omega_2^{1/2}\,\omega_3\,
    dp_2
  \Bigg\|_{L_p^\infty} \\[4pt]
&\lesssim
  \|\omega^{l+\frac{1}{2}} f\|_{L_p^q}\,
  \|\omega g\|_{L_p^\infty}\,
  \|\omega^{1/2} k\|_{L_p^1}.
\end{align*}
This completes the proof.
\end{proof}
\medskip

\begin{lemma}\label{lem:G2_bound}
For any \(1 \le q \le \infty\) and \(l \in \mathbb{R}\), the trilinear form 
\(\mathcal{G}_2[f,g,k]\) satisfies the estimate
\begin{equation}\label{bound_G2}
    \|\omega^l \mathcal{G}_2[f,g,k]\|_{L_p^q}
    \lesssim
    \|\omega^{l+\frac{1}{2}} f\|_{L_p^q}\,
    \|\omega^{\frac{1}{2}}g\|_{L_p^1}\,
    \|\omega k\|_{L_p^\infty}.
\end{equation}
\end{lemma}
\begin{proof}
The proof follows the same argument as in Lemma~\ref{lem:L2_bound}, 
using the parametrization of the resonant manifold and the lower bound \eqref{F+lower}.
Starting from the definition of \(\mathcal{G}_2\) and proceeding as before, 
we obtain
\begin{align*}
\|\omega^l \mathcal{G}_2[f,g,k]\|_{L_p^q}
&\lesssim
  \Bigg\|
    \omega_0^{\,l+\frac{1}{2}}
    \int_0^{2\pi}
      f(p_0)\,g(p_2)\,k(p_3)\,
      \omega_1\,\omega_2^{1/2}\,\omega_3\,
    dp_2
  \Bigg\|_{L_p^q} \\[4pt]
&\le
  \|\omega^{l+\frac{1}{2}} f\|_{L_p^q}
  \Bigg\|
    \int_0^{2\pi}
      g(p_2)\,k(p_3)\,
      \omega_1\,\omega_2^{1/2}\,\omega_3\,
    dp_2
  \Bigg\|_{L_p^\infty} \\[4pt]
&\lesssim
  \|\omega^{l+\frac{1}{2}} f\|_{L_p^q}\,
  \|\omega^{1/2} g\|_{L_p^1}\,
  \|\omega k\|_{L_p^\infty}.
\end{align*}
This completes the proof.
\end{proof}

\medskip

\begin{lemma}\label{lem:L1_bound}
For any \(1 \le q \le \infty\) and \(l \in \mathbb{R}\), the trilinear form 
\(\mathcal{L}_1[f,g,k]\) satisfies the estimate
\begin{equation}\label{bound_L1}
    \|\omega^l \mathcal{L}_1[f,g,k]\|_{L_p^q}
    \lesssim
    \|\omega^{l+\frac{1}{2}} f\|_{L_p^q}\,
    \|\omega^{\frac{1}{2}} k\|_{L_p^1}\,
    \|\omega g\|_{L_p^\infty}.
\end{equation}
\end{lemma}

\begin{proof}
The proof is essentially identical to that of Lemma~\ref{lem:L2_bound}.
Indeed, the parametrization of the resonance manifold in this case involves
interchanging the roles of \(p_2\) and \(p_3\):
\[
\begin{cases}
p_1 = h(p_0,p_3),\\[2pt]
p_2 = p_0 + p_1 - p_3.
\end{cases}
\]
With this change of variables, the trilinear form becomes
\[
\mathcal{L}_1[f,g,k]
    = \int_0^{2\pi}
        \frac{\omega_0 \omega_1 \omega_2 \omega_3}
             {\sqrt{F_+(p_0,p_3)}}\,
        f\,g_1\,k_3 \, dp_3,
\]
which coincides with the structure of \(\mathcal{L}_2[f,g,k]\).
Therefore, the bound \eqref{bound_L1} follows directly from 
Lemma~\ref{lem:L2_bound}.
\end{proof}
\medskip
\begin{lemma}\label{lem:G2_alt_bound}
By interchanging the roles of \(p_2\) and \(p_3\) in the parametrization of the resonance manifold,
an alternative bound for \(\mathcal{G}_2[f,g,k]\) holds:
\begin{equation}\label{bound_G2_alt}
    \|\omega^l \mathcal{G}_2[f,g,k]\|_{L_p^q}
    \lesssim
    \|\omega^{l+\frac{1}{2}} f\|_{L_p^q}\,
    \|\omega^{\frac{1}{2}} k\|_{L_p^1}\,
      \|\omega g\|_{L_p^\infty}.
\end{equation}
The proof is identical to that of Lemma~\ref{lem:L1_bound}, since both rely on
the exchange of \(p_2\) and \(p_3\) in the parametrization and the symmetry of
\(\Sigma\) and \(\Omega\) with respect to these variables.
\end{lemma}

Now we establish a basic geometric property of the resonance manifold,
which will be useful in the subsequent analysis.
We also record the following useful identity, 
which appears in a similar form in \cite[Lemma~19]{GermainLaMenegaki2024RJ}.

\begin{lemma}\label{lem:resonance_identity}
Let \(p_0,p_2\in[0,2\pi]\), and let \(p_1,p_3\) be determined by
\[
\begin{cases}
    p_3 = p_0 + p_1 - p_2, \\[4pt]
    p_1 = h(p_0,p_2).
\end{cases}
\]
Set \(\omega_j=\omega(p_j)=|\sin(p_j/2)|\).
Then the following resonance bound holds:
\begin{equation}\label{eq:resonance_identity}
    |\omega_1\,\omega_2\,\omega_3|
    \;\lesssim\;
    \omega_0 .
\end{equation}
\end{lemma}

\medskip

\begin{lemma}\label{lem:G1_Linf_bound}
Let \(l \in \mathbb{R}\). The trilinear form \(\mathcal{G}_1[f,g,k]\) satisfies the following weighted $L_p^\infty$estimates:

\medskip

\noindent
(i) For \(l \ge -\tfrac{1}{2}\),
\begin{equation}\label{eq:G1_Linf_bound_case1}
    \|\omega^l \mathcal{G}_1[f,g,k]\|_{L_p^\infty}
    \lesssim
    \|\omega f\|_{L_p^\infty}\,
    \|\omega^{\frac{1}{2}} g\|_{L_p^1}\,
    \|\omega k\|_{L_p^\infty},
\end{equation}
and
\medskip

\noindent
(ii) For \(l < -\tfrac{1}{2}\),
\begin{equation}\label{eq:G1_Linf_bound_case3}
    \|\omega^l \mathcal{G}_1[f,g,k]\|_{L_p^\infty}
    \lesssim
    \|\omega^{l+\frac{3}{2}} f\|_{L_p^\infty}\,
    \|\omega^{l+1} g\|_{L_p^1}\,
    \|\omega^{l+\frac{3}{2}} k\|_{L_p^\infty},
\end{equation}
\end{lemma}

\begin{proof}


\medskip
\noindent
\textbf{Case 1:} \(l \ge -\tfrac{1}{2}\).  
Using the lower bound \eqref{F+lower}, we obtain
\begin{align*}
|\omega^l \mathcal{G}_1[f,g,k]|
&= \Bigg|\omega_0^l
    \int_0^{2\pi}
        \frac{\omega_0\omega_1\omega_2\omega_3}
             {\sqrt{F_+(p_0,p_2)}}\, f_1 g_2 k_3\, dp_2
  \Bigg|\\[4pt]
&\lesssim
  \int_0^{2\pi}
        \big|\omega_0^{l+\frac{1}{2}}\omega_1\omega_2^{\frac{1}{2}}\omega_3
        f_1 g_2 k_3\big|\, dp_2.
\end{align*}
Since \(\omega_0^{l+\frac{1}{2}}\le 1\) if \(l\ge -\frac{1}{2}\), we obtain 
\begin{align*}
|\omega^l \mathcal{G}_1[f,g,k]|
&\lesssim 
  \int_0^{2\pi}
      |\omega_1\omega_2^{\frac{1}{2}}\omega_3
       f_1 g_2 k_3|\, dp_2
\end{align*}
Applying Hölder’s inequality in $p_2$ and using the boundedness of the weights then yields
\begin{align*}
|\omega^l \mathcal{G}_1[f,g,k]|
&\lesssim
  \|\omega  f\|_{L_p^\infty}\,
  \|\omega^{\frac{1}{2}} g\|_{L_p^1}\,
  \|\omega k\|_{L_p^\infty}
  .
\end{align*}

\medskip
\noindent
\textbf{Case 2:} \(l < -\tfrac{1}{2}\).  
Using Lemma~\ref{lem:resonance_identity}, which gives 
\(|\omega_1\omega_2\omega_3|\lesssim \omega_0\),
and the lower bound \eqref{F+lower},
we obtain
\begin{align*}
|\omega^l \mathcal{G}_1[f,g,k]|
&\lesssim
  \int_0^{2\pi}
      |\omega_1^{l+\frac{3}{2}}\omega_2^{l+1}\omega_3^{l+\frac{3}{2}}
       f_1 g_2 k_3|\, dp_2\\[4pt]
&\le
  \|\omega^{l+\frac{3}{2}} f\|_{L_p^\infty}\,
  \|\omega^{l+1} g\|_{L_p^1}\,
  \|\omega^{l+\frac{3}{2}} k\|_{L_p^\infty}.
\end{align*}
This completes the proof.
\end{proof}
We recall the following trigonometric identity, which will be used later 
(see \cite[Lemma~19]{GermainLaMenegaki2024RJ}).

\begin{lemma}\label{lem:resonance_tan_identity}
Let \(p_0,p_2\in[0,2\pi]\), and let \(p_1,p_3\) be determined by
\[
\begin{cases}
    p_3 = p_0 + p_1 - p_2, \\[4pt]
    p_1 = h(p_0,p_2)
\end{cases}
\]
Set \(\omega_j=\omega(p_j)=|\sin(p_j/2)|\).
Then the following resonance identity holds:
\begin{equation}\label{eq:tan_identity}
    |\omega_1\,\omega_3|
    = \tan^2\!\left(\frac{p_2 - p_0}{4}\right)
      \omega_2\,\omega_0.
\end{equation}
\end{lemma}
\begin{lemma}\label{lem:G1_L1_bounds}
Let \(l \in \mathbb{R}\). The trilinear form \(\mathcal{G}_1[f,g,k]\) satisfies the following weighted \(L_p^1\) estimates.

\medskip
\noindent
(i) For \(l \ge -\frac{1}{2}\),
\begin{equation}\label{eq:G1_L1_case1}
    \|\omega^l \mathcal{G}_1[f,g,k]\|_{L_p^1}
    \lesssim
    \|f\|_{L_p^1}\,
    \|\omega^{\frac{1}{2}} g\|_{L_p^\infty}\,
    \|k\|_{L_p^1}.
\end{equation}
\medskip
\noindent
(ii) For \(l <-\frac{1}{2}\),
\begin{equation}\label{eq:G1_L1_case3}
    \|\omega^l \mathcal{G}_1[f,g,k]\|_{L_p^1}
    \lesssim
    \|\omega^{l+\frac{1}{2}} f\|_{L_p^1}\,
    \|\omega^{l+1} g\|_{L_p^\infty}\,
    \|\omega^{l+\frac{1}{2}} k\|_{L_p^1}.
\end{equation}
\end{lemma}

The proof of Lemma~\ref{lem:G1_L1_bounds} proceeds in three steps.
First, we compute the Jacobian associated with the change of variables
$(p_0,p_2)\mapsto(p_1,p_3)$ on the resonance manifold, which captures the geometric contribution precisely. 
Second, we establish a quantitative bound on the combined Jacobian–weight
factor, showing that the singularity introduced by the transformation remains
uniformly controlled.  
Finally, using this bound together with the resonance properties derived
earlier, we estimate $\mathcal{G}_1[f,g,k]$ in the weighted $L^1$ norm,
completing the proof.
\medskip

The most delicate part of the analysis concerns the trilinear form
$\mathcal{G}_1[f,g,k]$, which represents the dominant gain contribution
in the collision operator. 
Unlike $\mathcal{L}_1$, $\mathcal{L}_2$, or $\mathcal{G}_2$, the term
$\mathcal{G}_1$ exhibits a more intricate dependence on the resonance geometry,
and its singular structure requires a careful treatment of the associated
Jacobian factor.

\vspace{1em}
\begin{lemma}[Jacobian structure]\label{lem:jacobian_structure}
Consider the change of variables
\[
(p_0,\,p_2)\longmapsto (p_1,\,p_3),
\qquad \]
\text{with}\quad
\[\begin{cases}
p_3 = p_0 + p_1 - p_2, \\[3pt]
p_1 = h(p_0,p_2),
\end{cases}
\]
where $h$ is explicitly given in \eqref{eq:h_def}.
Then the Jacobian determinant satisfies
\begin{align}
\det J&:= \det\frac{\partial(p_1,p_3)}{\partial(p_0,p_2)}
  =\frac{\sin\!\big(\tfrac{p_0+p_2}{4}\big)\,
          \tan\!\big(\tfrac{|p_0-p_2|}{4}\big)}
        {\sqrt{1-\cos^2\!\big(\tfrac{p_0+p_2}{4}\big)
                  \tan^2\!\big(\tfrac{|p_0-p_2|}{4}\big)}}.
\label{eq:jacobian_formula}
\end{align}
\end{lemma}

\begin{proof}
A direct calculation gives that
\begin{align}
\det J
  &=\begin{vmatrix}
    \partial_{p_0} h(p_0,p_2)& \partial_{p_2} h(p_0,p_2) \\[1em]
    \partial_{p_0} h(p_0,p_2)+1& \partial_{p_2} h(p_0,p_2)-1
\end{vmatrix}\nonumber\\
   &=-\partial_{p_0}h(p_0,p_2)-\partial_{p_2}h(p_0,p_2).
\label{eq:jacobian_formula1}
\end{align}
Differentiating $h(p_0,p_2)$ defined in \eqref{eq:h_def}, one computes
\[
\partial_{p_0} h(p_0,p_2)
 = \frac{2\!\left(
 -\tfrac{1}{4}\sin(\frac{p_0+p_2}{4})\tan(\frac{|p_0-p_2|}{4})
 -\tfrac{(p_2-p_0)\cos\!\frac{p_0+p_2}{4}
   \sec^2\!\frac{|p_0-p_2|}{4}}{4|p_2-p_0|}
  \right)}
 {\sqrt{1-\cos^2\!\frac{p_0+p_2}{4}
  \tan^2\!\frac{|p_0-p_2|}{4}}}
 - \frac{1}{2},
\]
and similarly
\[
\partial_{p_2} h(p_0,p_2)
 = \frac{2\!\left(
  \tfrac{(p_2-p_0)\cos\!\frac{p_0+p_2}{4}
   \sec^2\!\frac{|p_0-p_2|}{4}}{4|p_2-p_0|}
  -\tfrac{1}{4}\sin(\frac{p_0+p_2}{4})\tan(\frac{|p_0-p_2|}{4})
 \right)}
 {\sqrt{1-\cos^2\!\frac{p_0+p_2}{4}
  \tan^2\!\frac{|p_0-p_2|}{4}}}
 + \frac{1}{2}.
\]
Substituting these into \eqref{eq:jacobian_formula1} gives the claimed closed form \eqref{eq:jacobian_formula}.
\end{proof}

\begin{lemma}[Jacobian-weight interaction]\label{lem:jacobian_bound}
For any \(l\in\mathbb{R}\) and $p_0,p_2\in[0,2\pi]$ on the resonance manifold,
the following bound holds:
\begin{equation}\label{eq:jacobian_weight_bound}
\left|
\frac{\omega_1\omega_3}
     {\det J}
\right|
\lesssim 
1.
\end{equation}
\end{lemma}

\begin{proof}
From Lemma~\ref{lem:jacobian_structure}, together with the elementary bound
established in \cite[Lemma~12]{GermainLaMenegaki2024RJ},
\[
\bigl|\cos\!\bigl(\tfrac{p_0+p_2}{4}\bigr)\,
       \tan\!\bigl(\tfrac{|p_0-p_2|}{4}\bigr)\bigr|
       \le 1, \qquad (p_0,p_2)\in [0,2\pi]^2,
\]
we obtain
\begin{equation*}
 \big|\det J\big|^{-1}
   \le \frac{1}{
   \sin(\frac{p_0+p_2}{4})\,
   \tan(\frac{|p_0-p_2|}{4})}.   
\end{equation*}
Hence
\begin{equation}\label{boundJ-1}
\left|
\frac{\omega_1\omega_3}
     {\det J}
\right|
\le
\left|
\frac{\omega_1\omega_3}{
  \sin(\frac{p_0+p_2}{4})\,
  \tan(\frac{|p_0-p_2|}{4})}
\right|.
\end{equation}
We claim that the last factor is uniformly bounded.  
Using the identity from Lemma~\ref{lem:resonance_tan_identity},
\[
|\omega_1\omega_3|
  = \tan^2\!\left(\tfrac{p_2-p_0}{4}\right)
    \omega_0\omega_2,
\]
we obtain
\begin{equation}\label{boundJ-12}
\left|
\frac{\omega_1\omega_3}{
  \sin(\frac{p_0+p_2}{4})\,
  \tan(\frac{|p_0-p_2|}{4})}
\right|
= 
\left|
\frac{\omega_0\omega_2\tan\!(\frac{p_0-p_2}{4})}
     {\sin(\frac{p_0+p_2}{4})}
\right|.
\end{equation}

Possible singularities may occur only when  
either $|\tan(\frac{p_0-p_2}{4})|=\infty$, corresponding to $(p_0,p_2)=(0,2\pi)$ or   $(2\pi,0) $,  
or $\sin(\frac{p_0+p_2}{4})=0$, corresponding to $(p_0,p_2)=(0,0)$ or $(2\pi,2\pi)$.

\medskip\noindent
\textbf{Case 1.} $(p_0,p_2)\approx(0,2\pi)$  or   $(2\pi,0) $.  By symmetry it suffices to consider $(0,2\pi)$.
Let $\epsilon=\omega_0$ and $s=\omega_2$.  
Following the computation in~\cite[Lemma~18]{GermainLaMenegaki2024RJ}\cite[Lemma~19]{GermainLaMenegaki2024RJ}, one has
\[
|\omega_1\omega_3| \sim \frac{\epsilon s}{\epsilon^2+s^2},\]
Using Lemma~\ref{lem:resonance_tan_identity}, we have
\[
|\tan(\tfrac{p_0-p_2}{4})|\sim (\epsilon^2+s^2)^{-1/2}.
\]
Therefore
\[
\left|
\frac{\omega_1\omega_3}{
 \sin(\frac{p_0+p_2}{4})\tan(\frac{|p_0-p_2|}{4})}
\right|
\lesssim \frac{\epsilon s}{\sqrt{\epsilon^2+s^2}}\lesssim 1.
\]

\medskip\noindent
\textbf{Case 2.} $(p_0,p_2)\approx(0,0)$ or $(2\pi,2\pi)$.  
By symmetry it suffices to consider $(0,0)$.  
Then 
\[
\omega_0\sim p_0,\quad \omega_2\sim p_2,\quad
\sin(\frac{p_0+p_2}{4})\sim p_0+p_2,\quad
\tan(\frac{|p_0-p_2|}{4})\sim |p_0-p_2|.
\]
Hence
\[
\left|
\frac{\omega_1\omega_3}{
 \sin(\frac{p_0+p_2}{4})\tan(\frac{|p_0-p_2|}{4})}
\right|
\sim \frac{p_0p_2|p_0-p_2|}{p_0+p_2}\lesssim 1.
\]
Combining two cases above, we proved the boundedness of $\left|
\omega_1\omega_3
  [\sin(\frac{p_0+p_2}{4})\,
  \tan(\frac{|p_0-p_2|}{4})]^{-1}
\right|$, which, together with \eqref{boundJ-1}, gives the desired estimate \eqref{eq:jacobian_weight_bound}.
\end{proof}

\begin{proof}[Proof of Lemma~\ref{lem:G1_L1_bounds}]
We start from  Fubini's theorem and the explicit form of $\mathcal{G}_1$:
\begin{align*}
\|\omega^l \mathcal{G}_1[f,g,k]\|_{L_p^1}
&= \Big\|\int_0^{2\pi}
   \frac{\omega_0^{l+1}\omega_1\omega_2\omega_3}
        {\sqrt{F_+(p_0,p_2)}}
        f_1 g_2 k_3\,dp_2\Big\|_{L_p^1}\\
&\le \int_{\mathbb{T}^2}
   \Big|
   \frac{\omega_0^{l+1}\omega_1\omega_2\omega_3}
        {\sqrt{F_+(p_0,p_2)}}
        f_1 g_2 k_3
   \Big|\,dp_0\,dp_2\\
&\lesssim
\int_{\mathbb{T}^2}
   |\omega_0^{l+\frac{1}{2}}\omega_1\omega_2^{\frac{1}{2}}\omega_3
     f_1 g_2 k_3|\,dp_0\,dp_2.
\end{align*}

Applying the change of variables $(p_0,p_2)\mapsto(p_1,p_3)$ and 
Lemma~\ref{lem:jacobian_bound}, we obtain
\[
\|\omega^l \mathcal{G}_1[f,g,k]\|_{L_p^1}
\lesssim
\int_{\mathbb{T}^2}
   |\omega_0^{l+\frac{1}{2}}\omega_2^{\frac{1}{2}}
     f_1 g_2 k_3|\,dp_1\,dp_3.
\]
The remainder of the proof proceeds by the same decomposition and 
frequency comparison arguments as in the $L_p^\infty$ case.

\medskip\noindent
\textbf{Case 1.} \(l\ge -\tfrac{1}{2}\). Observing \(\omega_0^{l+\frac{1}{2}}\le 1\) if \(l\ge -\frac{1}{2}\), we obtain

   \begin{align*}
     \|\omega^l \mathcal{G}_1[f,g,k]\|_{L_p^1}& \lesssim\int_0^{2\pi }\left|\omega_2^{\frac{1}{2}}f_1g_2k_3\right|\, dp_1\, dp_3\\    
     &\le \|f\|_{L_p^1}\|\omega^{\frac{1}{2}}g\|_{L_p^\infty}\|k\|_{L_p^1}.
     \end{align*}
      
\medskip\noindent
\textbf{Case 2.} \(l<-\tfrac{1}{2}\).  
Using Lemma~\ref{lem:resonance_identity}, which gives 
\(|\omega_1\omega_2\omega_3|\lesssim \omega_0\), we deduce
\begin{align*}
\|\omega^l \mathcal{G}_1[f,g,k]\|_{L_p^1}
&\lesssim\int_{\mathbb{T}^2}\left|\omega_1^{l+\frac{1}{2}}\omega_2^{l+1}\omega_3^{l+\frac{1}{2}}f_1g_2k_3\right|\, dp_1\, dp_3\\
&\le \|\omega^{l+\frac{1}{2}}f\|_{L_p^1}\,
\|\omega^{l+1}g\|_{L_p^\infty}\,
\|\omega^{l+\frac{1}{2}}k\|_{L_p^1}.
\end{align*}
This completes the proof.
\end{proof}

\begin{lemma}\label{lem:G1_L2_bound}
The trilinear form $\mathcal{G}_1[f,g,k]$ satisfies the following weighted
$L_p^2$ estimates.

\medskip
\noindent
\textnormal{(i)} For $\,l \ge -\tfrac{1}{2}$,
\begin{equation}\label{eq:G1_L2_case1}
\begin{aligned}
    \|\omega^l \mathcal{G}_1[f,g,k]\|_{L_p^2}
    &\lesssim
    \Big(
       \|\omega f\|_{L_p^2}\,
       \|\omega k\|_{L_p^\infty}
      +\|\omega k\|_{L_p^2}\,
       \|\omega f\|_{L_p^\infty}
    \Big)
    \|\omega^{\frac{1}{2}} g\|_{L_p^1}.
\end{aligned}
\end{equation}

\medskip
\noindent
\textnormal{(ii)} For $\,l < -\tfrac{1}{2}$,
\begin{equation}\label{eq:G1_L2_case3}
\begin{aligned}
    \|\omega^l \mathcal{G}_1[f,g,k]\|_{L_p^2}
    &\lesssim
    \Big(
       \|\omega^{l+\frac{3}{2}} f\|_{L_p^2}\,
       \|\omega^{l+\frac{3}{2}} k\|_{L_p^\infty}
      +\|\omega^{l+\frac{3}{2}} k\|_{L_p^2}\,
       \|\omega^{l+\frac{3}{2}} f\|_{L_p^\infty}
    \Big)
    \|\omega^{l+1} g\|_{L_p^1}.
\end{aligned}
\end{equation}
\end{lemma}


\begin{proof}[Proof of Lemma~\ref{lem:G1_L2_bound}]
We test against a function $\phi$ with $\|\phi\|_{L_p^2}=1$ and write
\begin{align*}
   \langle \omega^l\mathcal{G}_1[f,g,k],\phi\rangle
   &=\int_{\mathbb{T}^2}
     \frac{\omega_0^{l+1}\omega_1\omega_2\omega_3}
          {\sqrt{F_+(p_0,p_2)}}
     f_1 g_2 k_3 \phi_0\, dp_2\,dp_0\\
   &\lesssim
     \int_{\mathbb{T}^2}
     \big|\omega_0^{l+\frac{1}{2}}\omega_1\omega_2^{\frac{1}{2}}\omega_3
      f_1 g_2 k_3 \phi_0\big|\, dp_2\,dp_0,
\end{align*}
where we used the lower bound $\sqrt{F_+(p_0,p_2)}\gtrsim \omega_0^{1/2}\omega_2^{1/2}$.

\medskip
\noindent
\textbf{Case 1.} $l\ge -\tfrac{1}{2}$.  
Applying Hölder’s inequality and Fubini’s theorem gives
\begin{align*}
   |\langle \omega^l\mathcal{G}_1[f,g,k],\phi\rangle|
   &\lesssim
   \|\omega^{\frac{1}{2}} g\|_{L_p^1}\,
   \Big\|
   \int_0^{2\pi}
   |\omega_0^{l+\frac{1}{2}}\omega_1\omega_3
     f_1 k_3 \phi_0|\, dp_0
   \Big\|_{L_{p_2}^\infty}\\
   &\lesssim
   \|\omega^{\frac{1}{2}} g\|_{L_p^1}\,
   \|\phi\|_{L_p^2}\,
   \Big\|
   \Big(\int_0^{2\pi}
     |\omega_0^{2l+1}\omega_1^2\omega_3^2
      f_1^2 k_3^2|\, dp_0\Big)^{1/2}
   \Big\|_{L_{p_2}^\infty}.
\end{align*}

\medskip
\noindent
\emph{Step 1: Partition of the torus and change of variables.}  
Since \(\partial_{p_0} p_1 = \partial_{p_0} p_3 - 1\), we may decompose the integration domain
\(p_0 \in [0,2\pi]\) into two measurable subsets,
\[
[0,2\pi] = I_1 \cup I_2,
\]
such that the following estimates hold almost everywhere:
\[
\bigl|\partial_{p_0} p_1\bigr| \ge \frac{1}{2}
\quad \text{for } p_0 \in I_1,
\qquad
\partial_{p_0} p_3 \ge \frac{1}{2}
\quad \text{for } p_0 \in I_2 .
\]

On $I_1$ we perform the change of variable $p_0\mapsto p_1$,
while on $I_2$ we perform $p_0\mapsto p_3$.  

This implies
\begin{align*}
   \Big\|
   \Big(\int_0^{2\pi}
   |\omega_0^{2l+1}\omega_1^2\omega_3^2 f_1^2k_3^2|\,dp_0\Big)^{1/2}
   \Big\|_{L_{p_2}^\infty}
   &\lesssim
   \Big\|
   \Big(
      \int_{I_1}|\omega_0^{2l+1}\omega_1^2\omega_3^2 f_1^2k_3^2|\,dp_1
     +\int_{I_2}|\omega_0^{2l+1}\omega_1^2\omega_3^2 f_1^2k_3^2|\,dp_3
   \Big)^{1/2}
   \Big\|_{L_{p_2}^\infty}\\
   &\lesssim
   \Big\|
   \Big(
      \int_{\mathbb{T}}|\omega_0^{2l+1}\omega_1^2\omega_3^2 f_1^2k_3^2|\,dp_1
     +\int_{\mathbb{T}}|\omega_0^{2l+1}\omega_1^2\omega_3^2 f_1^2k_3^2|\,dp_3
   \Big)^{1/2}
   \Big\|_{L_{p_2}^\infty}.
\end{align*}

\medskip
\noindent
\emph{Step 2: Boundedness of $\omega_0^{2l+1}$.}  
Applying \(\omega_0^{2l+1}\le 1\), we obtain
\begin{align*}
   \Big\|
   \Big(\int_0^{2\pi}
   |\omega_0^{2l+1}\omega_1^2\omega_3^2 f_1^2k_3^2|\,dp_0\Big)^{1/2}
   \Big\|_{L_{p_2}^\infty}
   &\lesssim
   \|\omega f\|_{L_p^2}\|\omega k\|_{L_p^\infty}
   +\|\omega k\|_{L_p^2}\|\omega f\|_{L_p^\infty}
\end{align*}
Taking the supremum over $\phi$ with $\|\phi\|_{L_p^2}=1$ gives
\eqref{eq:G1_L2_case1}.

\medskip
\noindent
\textbf{Case 2.} $l<-\tfrac{1}{2}$.  
Using Lemma~\ref{lem:resonance_identity}, which implies
$|\omega_1\omega_2\omega_3|\lesssim\omega_0$, we obtain
\begin{align*}
   |\langle \omega^l\mathcal{G}_1[f,g,k],\phi\rangle|
   &\lesssim
   \int_{\mathbb{T}^2}
   |\omega_1^{l+\frac{3}{2}}\omega_2^{l+1}\omega_3^{l+\frac{3}{2}}
     f_1 g_2 k_3 \phi_0|\,dp_2\,dp_0\\
   &\lesssim
   \|\omega^{l+1}g\|_{L_p^1}\,
   \|\phi\|_{L_p^2}\,
   \Big\|
   \Big(\int_0^{2\pi}
     |\omega_1^{2l+3}\omega_3^{2l+3} f_1^2k_3^2|\,dp_0\Big)^{1/2}
   \Big\|_{L_{p_2}^\infty}.
\end{align*}
Repeating the same partition and change-of-variable argument as above, we find
\begin{align*}
   \Big\|
   \Big(\int_0^{2\pi}
     |\omega_1^{2l+3}\omega_3^{2l+3} f_1^2k_3^2|\,dp_0\Big)^{1/2}
   \Big\|_{L_{p_2}^\infty}
   &\lesssim 
     \|\Big(\int_0^{2\pi }
      \left|\omega_1^{2l+3}\omega_3^{2l+3}
       f_1^2k_3^2 \right|\,dp_1
       +
       \int_0^{2\pi }
      \left|\omega_1^{2l+3}\omega_3^{2l+3}
       f_1^2k_3^2 \right|\,dp_3
       \Big)^{\frac{1}{2}}
    \|_{L_{p_2}^\infty}\\
   &\lesssim
   \|\omega^{l+\frac{3}{2}}f\|_{L_p^2}\|\omega^{l+\frac{3}{2}}k\|_{L_p^\infty}
   +\|\omega^{l+\frac{3}{2}}k\|_{L_p^2}\|\omega^{l+\frac{3}{2}}f\|_{L_p^\infty}.
\end{align*}
Combining the above bounds yields \eqref{eq:G1_L2_case3}, which completes the proof.
\end{proof}

\begin{remark}[Comparison with the $L^1$ bound and the necessity of the refined $L^2$ estimate]
The $L^1$ estimate relies on a single change of variables 
\((p_0,p_2)\mapsto(p_1,p_3)\).  
In contrast, the $L^2$ bound requires a finer analysis: the torus is divided
into two regions \(I_1\) and \(I_2\), on which we separately use the
transformations \(p_0\mapsto p_1\) and \(p_0\mapsto p_3\).  
This partition ensures that the map is non-degenerate on each region and
allows the change of variables to be applied in a controlled fashion.

\medskip
If one instead attempts to use the embedding $L_p^2\hookrightarrow L_p^1$ in Lemma~\ref{lem:weight-properties} and
apply the same change of variables as in the $L^1$ estimate,
the resulting bound would be for $l\ge -\frac{1}{2}$
\[
   \|\omega^l\mathcal{G}_1[f,g,k]\|_{L_p^2}
   \lesssim
   \|\omega ^{\frac{1}{2}}f\|_{L_p^2}\,
   \|\omega^{\frac{1}{2}}k\|_{L_p^2}\,
   \|\omega^{\frac{1}{2}+\frac{1}{2}_{-}}g\|_{L_p^\infty},
\]
which, although valid, is too weak to control the nonlinear iteration required
for long-time existence. In particular, this simple $L^2L^2L^\infty$ structure
does not provide sufficient decay or weight redistribution to close the
bootstrap argument in the dispersive regime.

\medskip

Lemma~\ref{lem:G1_L2_bound} furnishes the strengthened weighted $L^2$
estimate needed to overcome this limitation.  
The localized transformations on \(I_1\) and \(I_2\), together with a more
efficient use of the $\omega$-weights, provide the additional smallness needed
to close the nonlinear bounds in \(L_x^r L_p^q\) and extend the lifespan from
\(\varepsilon^{-2}\) to \(\varepsilon^{-4}\).

\end{remark}
\begin{proposition}[Weighted bounds for the collision trilinear forms]\label{prop:master_bounds}
Let $l\in\mathbb R$ and $1\le q\le\infty$. Then the gain/loss operators satisfy:
\begin{align*}
&\textbf{Loss: } 
\|\omega^l\mathcal L_2[f,g,k]\|_{L_p^q}\lesssim 
\|\omega^{l+\frac12}f\|_{L_p^q}\,\|\omega g\|_{L_p^\infty}\,\|\omega^{\frac12}k\|_{L_p^1},\\
&\|\omega^l\mathcal L_1[f,g,k]\|_{L_p^q}\lesssim 
\|\omega^{l+\frac12}f\|_{L_p^q}\,\|\omega^{\frac12}k\|_{L_p^1}\,\|\omega g\|_{L_p^\infty},\\[2mm]
&\textbf{Gain (G2): }
\|\omega^l\mathcal G_2[f,g,k]\|_{L_p^q}\lesssim 
\|\omega^{l+\frac12}f\|_{L_p^q}\,\|\omega^{\frac12}g\|_{L_p^1}\,\|\omega k\|_{L_p^\infty},\\
&\|\omega^l\mathcal G_2[f,g,k]\|_{L_p^q}\lesssim 
\|\omega^{l+\frac12}f\|_{L_p^q}\,\|\omega g\|_{L_p^\infty}\,\|\omega^{\frac12}k\|_{L_p^1},\\[2mm]
&\textbf{Gain (G1), $L_p^\infty$: } \text{(cite Lemma \ref{lem:G1_Linf_bound})},\\
&\textbf{Gain (G1), $L_p^1$: } \text{(cite Lemma \ref{lem:G1_L1_bounds})},\\
&\textbf{Gain (G1), $L_p^2$: } \text{(cite Lemma \ref{lem:G1_L2_bound})}.
\end{align*}
All implicit constants depend only on the dispersion relation and absolute geometric constants.
\end{proposition}

In view of the monotonicity of weighted norms (see~\eqref{eq:weight-monotonicity}),
\[
   \|\omega^{l_1} f\|_{L_p^q} \le \|\omega^{l_2} f\|_{L_p^q}, 
   \qquad l_2 \ge l_1,
\]
we can collect the previous bounds for the gain and loss operators into a unified form.

\begin{proposition}[Unified weighted estimates for the collision operators]
Let $\mathcal{T}\in\{\mathcal{L}_1,\mathcal{L}_2,\mathcal{G}_1,\mathcal{G}_2\}$.
Then the following weighted $L_p^q$ bounds hold.

\medskip
\noindent
\textnormal{(i)} For $\,l\ge -\tfrac{1}{2}$,
\begin{align}
\|\omega^l \mathcal{T}[f,f,f]\|_{L_p^\infty}
 &\lesssim 
 \|\omega^{\min\{l+\frac12,\,1\}} f\|_{L_p^\infty}\,
 \|\omega f\|_{L_p^\infty}\,
 \|\omega^{\frac12} f\|_{L_p^1}, \label{eq:T_infty}\\[3pt]
\|\omega^l \mathcal{T}[f,f,f]\|_{L_p^1}
 &\lesssim 
 \| f\|_{L_p^1}\,
 \|\omega^{\frac12} f\|_{L_p^\infty}\,
 \|f\|_{L_p^1}, \label{eq:T_L1}\\[3pt]
\|\omega^l \mathcal{T}[f,f,f]\|_{L_p^2}
 &\lesssim 
 \|\omega^{\min\{l+\frac12,\,1\}} f\|_{L_p^2}\,
 \|\omega f\|_{L_p^\infty}\,
 \|\omega^{\frac12} f\|_{L_p^1}. \label{eq:T_L2}
\end{align}

\medskip
\noindent
\textnormal{(ii)} For $\,l< -\tfrac{1}{2}$,
\begin{align}
\|\omega^l \mathcal{T}[f,f,f]\|_{L_p^\infty}
 &\lesssim
 \|\omega^{l+\frac{1}{2}} f\|_{L_p^\infty}\,
 \|\omega^{l+\frac{3}{2}} f\|_{L_p^\infty}\,
 \|\omega^{l+1} f\|_{L_p^1}, \label{eq:T_low_infty}\\[3pt]
\|\omega^l \mathcal{T}[f,f,f]\|_{L_p^1}
 &\lesssim
 \|\omega^{l+\frac{1}{2}} f\|_{L_p^1}\,
 \|\omega^{l+1} f\|_{L_p^\infty}\,
 \|\omega^{l+\frac{1}{2}} f\|_{L_p^1}, \label{eq:T_low_L1}\\[3pt]
\|\omega^l \mathcal{T}[f,f,f]\|_{L_p^2}
 &\lesssim
 \|\omega^{l+\frac{1}{2}} f\|_{L_p^2}\,
 \|\omega^{l+\frac{3}{2}} f\|_{L_p^\infty}\,
 \|\omega^{l+1} f\|_{L_p^1}. \label{eq:T_low_L2}
\end{align}
\end{proposition}

By the embedding of weighted $L^q$ spaces established in Lemma~\ref{lem:weight-properties}, we have
\[
\|\omega^{\frac{1}{2}} f\|_{L_p^1}
\lesssim
\|\omega^{\frac{1}{2}+1_-} f\|_{L_p^\infty}.
\]
Combining this estimate with the monotonicity property of the weighted norms, 
we can recover the same weighted $L_p^\infty$ bound as derived in 
\cite[Sec.~3.1]{GermainLaMenegaki2024RJ}.
\begin{proposition}\label{prop:L_infty_bound}
Let $\mathcal{T}\in\{\mathcal{L}_1,\mathcal{L}_2,\mathcal{G}_1,\mathcal{G}_2\}$.
Then the following weighted $L_p^\infty$ estimate holds: for any $\alpha\le 1$,
\begin{align}\label{L_inf_bound}
\|\omega^\alpha \mathcal{T}[f,g,k]\|_{L_p^\infty}
\lesssim 
\|\omega^\alpha f\|_{L_p^\infty}\,
\|\omega^\alpha g\|_{L_p^\infty}\,
\|\omega^\alpha k\|_{L_p^\infty}.
\end{align}
\end{proposition}

\section{Long-Time Existence}
From this point on, we introduce the following weighted dispersive norm, which combines the uniform-in-time control in 
$L_{x,p}^\infty(\omega^\alpha)$ with the decay-in-time estimates arising from the dispersive behavior of the transport semigroup:
\begin{align}\label{def:X_norm}
\|f\|_{X}
:=
\sup_{t\ge 0}\|\omega^{\alpha} f(t)\|_{L_{x,p}^\infty}
+\sup_{t\ge 0}\|f(t)\|_{L_{x,p}^1}
+\sup_{t\ge 0}\langle t\rangle^{\frac{1}{2}}\|\omega^{\beta} f(t)\|_{L_x^2L_p^1}
+\sup_{t\ge 0}\langle t\rangle^{\frac{1}{2}}\|\omega^{\gamma} f(t)\|_{L_x^\infty L_p^2}.
\end{align}
This mixed norm captures both the conservative (energy-type) and the dispersive aspects of the solution, 
and will serve as the main functional setting for the contraction argument in the subsequent analysis.

\medskip
\noindent
\qedhere

\begin{theorem}[Local well-posedness in weighted dispersive norms]\label{thm:LWP_dispersive}
Assume the weight exponents $\alpha,\beta,\gamma$ satisfy the structural conditions
\[
\alpha\le \tfrac12,
\qquad 
\alpha<\gamma+\tfrac12,
\qquad
\alpha<2\beta+1.
\]
Then there exists $0<\varepsilon_0\ll 1$ such that for all 
$0<\varepsilon<\varepsilon_0$, if the initial datum satisfies the small weighted
$L^\infty\cap L^1$ condition
\[
\|\omega^{\alpha} f_0\|_{L_{x,p}^\infty}\le \varepsilon,
\qquad 
\|f_0\|_{L_{x,p}^1}\le \varepsilon,
\]
the inhomogeneous kinetic FPU equation admits a unique mild solution
\[
f\in C\big([0,T];\,L_{x,p}^\infty(\omega^\alpha)\cap L_{x,p}^1\big),
\qquad
f(t)=S(t)f_0+\int_0^t S(t-s)\mathcal{C}[f](s)\,ds,
\]
for some $T\sim 1$ with $T>1$.
In particular, since $T\sim1$, the dispersive weights
$\langle t\rangle^{1/2}$ remain uniformly comparable to~$1$ on~$[0,T]$.

Moreover, the solution satisfies the uniform bound
\begin{equation}\label{eq:X_norm_estimate1}
\|f\|_{X}
:=
\sup_{t\in[0,T]}\|\omega^{\alpha} f(t)\|_{L_{x,p}^\infty}
+\sup_{t\in[0,T]}\|f(t)\|_{L_{x,p}^1}
+\sup_{t\in[0,T]}\langle t\rangle^{\frac12}\|\omega^{\beta} f(t)\|_{L_x^2L_p^1}
+\sup_{t\in[0,T]}\langle t\rangle^{\frac12}\|\omega^{\gamma} f(t)\|_{L_x^\infty L_p^2}
\le C_0\varepsilon,
\end{equation}
for some constant $C_0>0$ independent of $\varepsilon$.
\end{theorem}

\begin{proof}

\noindent\textbf{Step 1. Contraction argument in $L^\infty_{x,p}(\omega^\alpha)\cap L^1_{x,p}$.}
We first show that the Duhamel mapping
\[
\Phi[f](t)
:= S(t) f_0 + \int_0^t S(t-s)\,\mathcal{C}[f](s)\,ds
\]
defines a contraction on the Banach space
\[
B_{4\varepsilon}
:=\Big\{
f\in C\big([0,T];\,L_{x,p}^\infty(\omega^\alpha)\cap L_{x,p}^1\big)\ :\
\sup_{t\in[0,T]}\|f(t)\|_{Y}\le 4\varepsilon
\Big\},
\]
where
\[
\|f\|_{Y}
:= \|f\|_{L_{x,p}^\infty(\omega^\alpha)}
   + \|f\|_{L_{x,p}^1}.
\]
This contraction property holds whenever
\[
\varepsilon < \varepsilon_0 \ll 1,
\qquad
T \lesssim \varepsilon^{-2}.
\]

In the proof of Theorem~\ref{Theorem:LWP_L_infinity}, we show that $\Phi$ maps the
$L^\infty_{x,p}(\omega^\alpha)$-ball of radius $2\varepsilon$ into itself and is a strict contraction.
More precisely,
\[
\|\Phi[f]\|_{L_{x,p}^\infty(\omega^\alpha)} \le 2\varepsilon,
\qquad
\|\Phi[f]-\Phi[g]\|_{L_{x,p}^\infty(\omega^\alpha)}
\le \frac14\,\|f-g\|_{L_{x,p}^\infty(\omega^\alpha)},
\]
whenever
\[
T \lesssim \varepsilon^{-2}, 
\qquad 
\alpha \le 1.
\]

Using Minkowski’s inequality, the dispersive estimate (Lemma~\ref{lem:SW-weighted}),
the $L^1$-bound of the collision operator~\eqref{eq:T_L1}, and the monotonicity and
embedding of weighted norms~\eqref{eq:weight-monotonicity}–\eqref{eq:weighted_embedding}, we obtain
\begin{align*}
\|\Phi[f]\|_{L_{x,p}^1}
&\le \|\mathcal{S}(t)f_0\|_{L_{x,p}^1}
+\Big\|\int_0^t \mathcal{S}(t-s)\mathcal{C}[f](s)\,ds\Big\|_{L_{x,p}^1}\\
&\le \|f_0\|_{L_{x,p}^1}
+\int_0^t \|\mathcal{C}[f](s)\|_{L_{x,p}^1}\,ds\\
&\le \varepsilon
+CT\|f\|_{L_{x,p}^1}\|\omega^{\frac12}f\|_{L_{x,p}^\infty}\|\omega^{1_-}f\|_{L_{x,p}^\infty}\\
&\le \varepsilon
+CT\|f\|_{L_{x,p}^1}\|\omega^{\alpha}f\|_{L_{x,p}^\infty}^2\\
&\le \varepsilon+CT\varepsilon^3,
\end{align*}
provided that $\alpha\le\tfrac{1}{2}$.  
Hence,
\[
\|\Phi[f]\|_{L_{x,p}^1}\le 2\varepsilon,
\qquad T\lesssim\varepsilon^{-2}.
\]

Using Minkowski’s inequality, the dispersive estimate (Lemma~\ref{lem:SW-weighted}),
the $L^1$-bounds for the collision trilinear forms in Proposition~\ref{prop:master_bounds},
and the monotonicity and embedding of weighted norms~\eqref{eq:weight-monotonicity}–\eqref{eq:weighted_embedding}, we obtain
\begin{align*}
\|\Phi[f]-\Phi[g]\|_{L_{x,p}^1}
&\le \Big\|\int_0^t \mathcal{S}(t-s)\big(\mathcal{C}[f]-\mathcal{C}[g]\big)(s)\,ds\Big\|_{L_{x,p}^1}\\
&\le \int_0^t \|\mathcal{C}[f]-\mathcal{C}[g]\|_{L_{x,p}^1}(s)\,ds\\
&\le \sum_{\Gamma\in\{\mathcal L_1,\mathcal L_2,\mathcal G_1,\mathcal G_2\}}
\int_0^t 
\Big(
\|\Gamma[f-g,f,f](s)\|_{L_{x,p}^1}
+\|\Gamma[g,f-g,f](s)\|_{L_{x,p}^1}
+\|\Gamma[g,g,f-g](s)\|_{L_{x,p}^1}
\Big)\,ds\\
&\le CT\sum_{k_1,k_2\in\{f,g\}}
\Big(
\|f-g\|_{L_{x,p}^1}\|\omega^{\frac12}k_1\|_{L_{x,p}^\infty}\|\omega^{1_-}k_2\|_{L_{x,p}^\infty}
+
\|\omega^{\frac12}(f-g)\|_{L_{x,p}^\infty}\|k_1\|_{L_{x,p}^1}\|\omega^{1_-}k_2\|_{L_{x,p}^\infty}
\Big)\\
&\le CT\sum_{k_1,k_2\in\{f,g\}}
\Big(
\|f-g\|_{L_{x,p}^1}\|\omega^\alpha k_1\|_{L_{x,p}^\infty}\|\omega^\alpha k_2\|_{L_{x,p}^\infty}
+
\|\omega^\alpha(f-g)\|_{L_{x,p}^\infty}\|k_1\|_{L_{x,p}^1}\|\omega^\alpha k_2\|_{L_{x,p}^\infty}
\Big)\\
&\le CT\varepsilon^2\|f-g\|_Y,
\end{align*}
provided that $\alpha\le \tfrac12$.  
Thus,
\[
\|\Phi[f]-\Phi[g]\|_{L_{x,p}^1}
\le \frac14\|f-g\|_{Y},
\qquad T\lesssim\varepsilon^{-2}.
\]

\medskip
Next we restrict to short time $T\sim1$ so that the dispersive bounds apply.

\medskip
\noindent\textbf{Step 2: $L_x^\infty L_p^2$-bound.}
By the monotonicity and embedding of weighted norms~\eqref{eq:weight-monotonicity}–\eqref{eq:weighted_embedding},
\[
\|\omega^\gamma f\|_{L_x^\infty L_p^2}
\lesssim \|\omega^{\gamma+\frac12_-}f\|_{L_{x,p}^\infty}
\le \|\omega^\alpha f\|_{L_{x,p}^\infty}
\le 2\varepsilon,
\]
as long as $\alpha<\gamma+\tfrac{1}{2}$.

\medskip
\noindent\textbf{Step 3: $L_x^2 L_p^1$-bound.}
Using the Cauchy--Schwarz inequality together with the monotonicity and embedding of weighted norms~\eqref{eq:weight-monotonicity}–\eqref{eq:weighted_embedding},
\begin{align*}
\|\omega^{\beta} f\|_{L_x^2L_p^1}
&\lesssim \|\omega^{\beta+\frac12_-}f\|_{L_{x,p}^2}
\le \|f\|_{L_{x,p}^1}^{\frac12}
\|\omega^{2\beta+1_-}f\|_{L_{x,p}^\infty}^{\frac12}\\[0.3em]
&\le (2\varepsilon)^{1/2}(2\varepsilon)^{1/2}
\le 2\varepsilon,
\end{align*}
provided that $\alpha<2\beta+1$.

\medskip
Combining the above bounds yields~\eqref{eq:X_norm_estimate1}, completing the proof.

\end{proof}

Motivated by the weighted dispersive norm and the corresponding dispersive estimates for the transport semigroup, we obtain a long-time existence result.
              
\begin{theorem}[Long-time well-posedness in weighted dispersive norms]\label{thm:Longtime_dispersive}
Assume the weight exponents $\alpha,\beta,\gamma$ satisfy the structural conditions
either
\[
-\tfrac{1}{2}\le\alpha\le \min\{\gamma,\tfrac{1}{2}\},\qquad 0\le \gamma<\tfrac{1}{2},
\quad\text{or}\quad
\alpha<-\tfrac{1}{2},\qquad 0\le \gamma\le \min\{\tfrac{1}{2},\alpha+\tfrac{3}{2}\},
\]
Then there exists $\varepsilon_0\in(0,1)$ such that for every $0<\varepsilon<\varepsilon_0$, if the initial datum satisfies
\[
\|\omega^{\alpha} f_0\|_{L_{x,p}^\infty}\le \varepsilon,\qquad 
\|f_0\|_{L_{x,p}^1}\le \varepsilon,\qquad
\|\omega^{\gamma-\frac{1}{2}}f_0\|_{L_x^2L_p^\infty}\le \varepsilon,\qquad
\|f_0\|_{L_x^1 L_p^2}\le\varepsilon,
\]
then the inhomogeneous kinetic FPU equation admits a unique mild solution
\[
f\in C([0,T];L_{x,p}^\infty(\omega^\alpha)\cap L_{x,p}^1),\qquad
f(t)=S(t)f_0+\int_0^t S(t-s)\,\mathcal{C}[f](s)\,ds,
\]
for some $T\sim \varepsilon^{-4}$. Moreover, the solution satisfies the uniform bound
\begin{equation}\label{eq:X_norm_estimate}
\|f\|_{X}
:=
\sup_{t\in[0,T]}\|\omega^{\alpha} f(t)\|_{L_{x,p}^\infty}
+\sup_{t\in[0,T]}\|f(t)\|_{L_{x,p}^1}
+\sup_{t\in[0,T]}\langle t\rangle^{\frac12}\|\omega^{\frac{1}{2}} f(t)\|_{L_x^2L_p^1}
+\sup_{t\in[0,T]}\langle t\rangle^{\frac12}\|\omega^{\gamma} f(t)\|_{L_x^\infty L_p^2}
\le C_0\,\varepsilon,
\end{equation}
for some constant $C_0>0$ independent of $\varepsilon$.
\end{theorem}

\begin{proof}
It suffices to close a bootstrap. Assume that
\begin{align}\label{bootstrap_condition}
       \max\Big\{ \sup_{t\in[0,T]}\|\omega^{\alpha} f(t)\|_{L_{x,p}^\infty},
\ \sup_{t\in[0,T]}\|f(t)\|_{L_{x,p}^1},
\ \sup_{t\in[0,T]}\langle t\rangle^{\frac12}\|\omega^{\frac{1}{2}} f(t)\|_{L_x^2L_p^1},
\ \sup_{t\in[0,T]}\langle t\rangle^{\frac12}\|\omega^{\gamma} f(t)\|_{L_x^\infty L_p^2}
\Big\}
\le C_1\varepsilon,
\end{align}
for some \(C_1>6\). Then
\[
\max\Big\{ \sup_{t\in[0,T]}\|\omega^{\alpha} f(t)\|_{L_{x,p}^\infty},
\ \sup_{t\in[0,T]}\|f(t)\|_{L_{x,p}^1},
\ \sup_{t\in[0,T]}\langle t\rangle^{\frac12}\|\omega^{\frac{1}{2}} f(t)\|_{L_x^2L_p^1},
\ \sup_{t\in[0,T]}\langle t\rangle^{\frac12}\|\omega^{\gamma} f(t)\|_{L_x^\infty L_p^2}
\Big\}
\le \frac{5}{6}C_1\varepsilon,
\]
provided \(T\lesssim \varepsilon^{-4}\).

By local well-posedness, we may assume \(t>1\).

\medskip
\noindent\textbf{Step 1: \(L^\infty_{x,p}\)-bound.}
Using Duhamel’s formula and Minkowski’s inequality,
\begin{align*}
    \|\omega^\alpha f\|_{L_{x,p}^\infty}\le 
    \|\omega^\alpha\mathcal{S}(t)f_0\|_{L_{x,p}^\infty}
    +\int_0^t \|\omega^\alpha \mathcal{S}(t-s)\mathcal{C}[f](s)\|_{L_{x,p}^\infty}\,ds.
\end{align*}
Since the weight \(\omega(p)\) commutes with \(\mathcal{S}(t)\) and by the dispersive estimate (see Lemma~\ref{lem:SW-weighted}),
\begin{align*}
    \|\omega^\alpha\mathcal{S}(t)f_0\|_{L_{x,p}^\infty}
   \le \|\omega^{\alpha}f_0\|_{L_{x,p}^\infty}\le \varepsilon<\frac{5}{12}C_1\varepsilon,
\end{align*}
whenever \(C_1>6\). Moreover,
\begin{align*}
    \int_0^t \|\omega^\alpha \mathcal{S}(t-s)\mathcal{C}[f](s)\|_{L_{x,p}^\infty}\,ds\lesssim \int_0^t \|\omega^\alpha \mathcal{C}[f](s)\|_{L_{x,p}^\infty}\,ds.
\end{align*}
If \(\alpha\ge -\tfrac{1}{2}\), then using the \(L^\infty\)-bound for the collision operator~\eqref{eq:T_infty}, the weighted \(L^q\) embedding~\eqref{eq:weighted_embedding}, the monotonicity of weighted norms~\eqref{eq:weight-monotonicity}, and the bootstrap hypothesis \eqref{bootstrap_condition}, we obtain
\begin{align*}
     \int_0^t \|\omega^\alpha \mathcal{C}[f](s)\|_{L_{x,p}^\infty}\,ds&\le C
     \int_0^t\|\omega^{\min\{\alpha+\frac{1}{2},1\}}f\|_{L_{x,p}^\infty}\|\omega f\|_{L_{x,p}^\infty}\|\omega^{\frac{1}{2}}f\|_{L_x^\infty L_p^1}\,ds
     \\
     &\le C
      \int_0^t\|\omega^{\min\{\alpha+\frac{1}{2},1\}}\|_{L_{x,p}^\infty}\|\omega f\|_{L_{x,p}^\infty}\|\omega^{\frac{1}{2}+\frac{1}{2}_{-}}f\|_{L_x^\infty L_p^2}\,ds
      \\
      &\le C\int_0^t \|\omega^\alpha f\|_{L_{x,p}^\infty}^2\|\omega^\gamma f\|_{L_x^\infty L_p^2}\,ds\\
      &\le C C_1^3\int_0^t \varepsilon^{3}\langle s\rangle^{-\frac{1}{2}}\,ds\\
     & \le C C_1^3 \varepsilon^{3} T^{\frac{1}{2}},
\end{align*}
provided \(\alpha \le \tfrac{1}{2}\), \(\gamma<1\), \(t\in(1,T]\). Hence
\[
 \int_0^t \|\omega^\alpha \mathcal{S}(t-s)\mathcal{C}[f](s)\|_{L_{x,p}^\infty}\,ds<\frac{5}{12}C_1 \varepsilon
\]
as long as \(t\in (1,T]\) and
\[
 C C_1 \varepsilon^3T^{\frac{1}{2}}\le \frac{5}{12}C_1\varepsilon,
\]
that is,
\[
 T\lesssim\varepsilon^{-4}.
\]
If \(\alpha<-\tfrac{1}{2}\), then using the \(L^\infty\)-bound~\eqref{eq:T_low_infty}, the monotonicity and embedding of weighted norms~\eqref{eq:weight-monotonicity}–\eqref{eq:weighted_embedding}, and \eqref{bootstrap_condition}, we obtain
\begin{align*}
    \int_0^t \|\omega^\alpha \mathcal{S}(t-s)\mathcal{C}[f](s)\|_{L_{x,p}^\infty}\,ds
    \le C
    \int_0^t\|\omega^{\alpha +\frac{1}{2}}f\|_{L_{x,p}^\infty}
    \|\omega^{\alpha +\frac{3}{2}}f\|_{L_{x,p}^\infty}
    \|\omega^{\alpha +1+\frac{1}{2}_{-}}f\|_{L_x^\infty L_p^2}\,ds
    \le CC_1^3 \varepsilon^3 T^{\frac{1}{2}},
\end{align*}
provided \(\gamma<\alpha+\tfrac{3}{2}\). In particular,
\[
 \int_0^t \|\omega^\alpha \mathcal{S}(t-s)\mathcal{C}[f](s)\|_{L_{x,p}^\infty}\,ds<\frac{5}{12} C_1\varepsilon
\]
whenever \(T\lesssim\varepsilon^{-4}\).

\medskip
\noindent\textbf{Step 2: \(L^1_{x,p}\)-bound.}
By Duhamel’s formula and Minkowski’s inequality,
\begin{align*}
    \| f\|_{L_{x,p}^1}\le 
    \|\mathcal{S}(t)f_0\|_{L_{x,p}^1}
    +\int_0^t \|\mathcal{S}(t-s)\mathcal{C}[f](s)\|_{L_{x,p}^1}\,ds.
\end{align*}
By the dispersive estimate (Lemma~\ref{lem:SW-weighted}),
\begin{align*}
     \|\mathcal{S}(t)f_0\|_{L_{x,p}^1}\le \|f_0\|_{L_{x,p}^1}\le \varepsilon<\frac{5}{12}C_1\varepsilon,
\end{align*}
for \(C_1>6\). Moreover,
\begin{align*}
    \int_0^t \|\mathcal{S}(t-s)\mathcal{C}[f](s)\|_{L_{x,p}^1}\,ds
    \lesssim
    \int_0^t \|\mathcal{C}[f](s)\|_{L_{x,p}^1}\,ds.
\end{align*}
Using the \(L^1\)-bound~\eqref{eq:T_L1}, the monotonicity and embedding of weighted norms~\eqref{eq:weight-monotonicity}–\eqref{eq:weighted_embedding}, and \eqref{bootstrap_condition}, we find
\begin{align*}
     \int_0^t \|\mathcal{C}[f](s)\|_{L_{x,p}^1}\,ds&\lesssim
     \int_0^t \|f\|_{L_{x,p}^1}\|\omega^{\frac{1}{2}}f\|_{L_{x,p}^\infty}
     \|f\|_{L_x^\infty L_p^1}\,ds\\
    & \le C
     \int_0^t \|f\|_{L_{x,p}^1}\|\omega^{\frac{1}{2}}f\|_{L_{x,p}^\infty}
     \|\omega^{\frac{1}{2}_{-}}f\|_{L_x^\infty L_p^2}\,ds\\
     &\le C
      \int_0^t \|f\|_{L_{x,p}^1}\|\omega^{\alpha}f\|_{L_{x,p}^\infty}
     \|\omega^{\gamma}f\|_{L_x^\infty L_p^2}\,ds\\
     &\le CC_1^3\int_0^t\varepsilon^3\langle s\rangle^{-\frac{1}{2}}\,ds\\
     &\le CC_1^3 T^{\frac{1}{2}}\varepsilon^3,
\end{align*}
provided \(\alpha \le \tfrac{1}{2}\) and \(\gamma <\tfrac{1}{2}\). Hence
\[
 \int_0^t \|\mathcal{S}(t-s)\mathcal{C}[f](s)\|_{L_{x,p}^1}\,ds<\frac{5}{12}C_1\varepsilon
 \quad\text{if}\quad
 T\lesssim\varepsilon^{-4}.
\]

\medskip
\noindent\textbf{Step 3: \(L_x^\infty L_p^2\)-bound.}
By Duhamel’s formula and Minkowski’s inequality,

\begin{align*}
    \langle t\rangle^{\frac{1}{2}} \|\omega^{\gamma}f\|_{L_x^\infty L_p^2}
    \le  \langle t\rangle^{\frac{1}{2}} \|\omega^{\gamma}\mathcal{S}(t)f_0\|_{L_x^\infty L_p^2}+
     \langle t\rangle^{\frac{1}{2}}\int_0^t \|\omega^\gamma \mathcal{S}(t-s)\mathcal{C}[f](s)\|_{L_x^\infty L_p^2}\,ds.
\end{align*}
Using the commutation property of the weights with $\mathcal{S}(t)$~\eqref{eq:commute} 
together with the weighted dispersive estimate in Lemma~\ref{lem:SW-weighted},

\begin{align*}
     \langle t\rangle^{\frac{1}{2}} \|\omega^{\gamma}\mathcal{S}(t)f_0\|_{L_x^\infty L_p^2}\le 2\langle t\rangle^{\frac{1}{2}} t^{-\frac{1}{2}}\|\omega^{\gamma-\frac{1}{2}}f_0\|_{L_x^2L_p^\infty}
     \le 2 \times 2^{\frac 1 4}\varepsilon<\frac{5}{12}C_1\varepsilon,
\end{align*}
for \(C_1>6\) and \(t>1\).
Denoting
\[
I_1= \langle t\rangle^{\frac{1}{2}}\int_0^{t} \|\omega^\gamma \mathcal{S}(t-s)\mathcal{C}[f](s)\|_{L_x^\infty L_p^2}\,ds,
\]
by the commutation property of the weights with $\mathcal{S}(t)$~\eqref{eq:commute} and Lemma~\ref{lem:SW-weighted},
\begin{align*}
    I_1\lesssim  
    \langle t\rangle^{\frac{1}{2}}\int_0^{t}
    | t-s|^{-\frac{1}{2}}\|\omega^{\gamma-\frac{1}{2}}f\|_{L_x^2L_p^\infty}\,ds.
\end{align*}
Using \eqref{eq:T_infty}, \eqref{eq:weighted_embedding}, \eqref{eq:weight-monotonicity}, and \eqref{bootstrap_condition}, for \(\gamma\ge 0\) we get
\begin{align*}
    I_1&\le C 
    \langle t\rangle^{\frac{1}{2}}
    \int_0^{t}
     | t-s |^{-\frac{1}{2}}\|\omega^{\min\{\gamma,1\}}f\|_{L_{x,p}^\infty}\|\omega f\|_{L_{x,p}^\infty}
     \|\omega^{\frac{1}{2}}f\|_{L_x^2L_p^1}\,ds\\
    &\le C  \langle t\rangle^{\frac{1}{2}}\int_0^{t}
     | t-s|^{-\frac{1}{2}}\|\omega^{\alpha}f\|_{L_{x,p}^\infty}^2
     \|\omega^{\frac{1}{2}}f\|_{L_x^2L_p^1}\,ds\\
    &\le CC_1^3 \langle t\rangle^{\frac{1}{2}}\int_0^{t}
    \varepsilon^3  | t-s|^{-\frac{1}{2}}  |s|^{-\frac{1}{2}}\,ds
    \le C C_1^3 \langle t\rangle^{\frac{1}{2}}  \varepsilon^3,
\end{align*}
provided \(\gamma\le 1\) and \(\alpha\le \gamma\). (Here we used, for \(t>1\),
\(\int_0^{t}  | t-s|^{-\frac{1}{2}}  |s|^{-\frac{1}{2}}\,ds\lesssim 1\), obtained by splitting \([0,t]\) into \([0,t/2]\cup[t/2,t]\).)
Consequently,
\[
I_1<\frac{5}{12}C_1\varepsilon
\]
whenever \( C C_1 \varepsilon^3T^{\frac{1}{2}}\le \frac{5}{12}C_1\varepsilon\), i.e. \(T\lesssim \varepsilon^{-4}\).

\medskip
\noindent\textbf{Step 4: \(L_x^2 L_p^1\)-bound.}
By Duhamel’s formula and Minkowski’s inequality,

\begin{align*}
    \langle t\rangle^{\frac{1}{2}}\|\omega^{\frac{1}{2}}f\|_{L_x^2L_p^1}
    \le
     \langle t\rangle^{\frac{1}{2}}\|\omega^{\frac{1}{2}}\mathcal{S}(t)f_0\|_{L_x^2L_p^1}+
     \langle t\rangle^{\frac{1}{2}}\int_0^t \|\omega^{\frac{1}{2}}\mathcal{S}(t-s)\mathcal{C}[f](s)\|_{L_x^2L_p^1}\,ds.
\end{align*}
Since by the commutation property of the weights with $\mathcal{S}(t)$~\eqref{eq:commute} and by Lemma~\ref{lem:SW-weighted},
\begin{align*}
 \langle t\rangle^{\frac{1}{2}}\|\omega^{\frac{1}{2}}\mathcal{S}(t)f_0\|_{L_x^2L_p^1}\le 2 \langle t\rangle^{\frac{1}{2}}
 t^{-\frac{1}{2}}\|f_0\|_{L_x^1 L_p^2}  \le 2 \times 2^{\frac 1 4}\varepsilon<\frac{6}{12}C_1\varepsilon,
\end{align*}
for \(C_1>6\) and \(t>1\). 
Denoting
\[
 I_1=\langle t\rangle^{\frac{1}{2}}\int_0^{t} \|\omega^{\frac{1}{2}}\mathcal{S}(t-s)\mathcal{C}[f](s)\|_{L_x^2L_p^1}\,ds,
\]
By ~\eqref{eq:commute} and Lemma~\ref{lem:SW-weighted},
\begin{align*}
     I_1\lesssim \langle t\rangle^{\frac{1}{2}}\int_0^{t} | t-s|^{-\frac{1}{2}}\|\mathcal{C}[f](s)\|_{L_x^1L_p^2}\,ds.
\end{align*}
Using \eqref{eq:T_L2}, \eqref{eq:weighted_embedding}, \eqref{eq:weight-monotonicity}, and \eqref{bootstrap_condition},
\begin{align*}
    I_1&\le C
    \langle t\rangle^{\frac{1}{2}}\int_0^{t} | t-s|^{-\frac{1}{2}}\|\omega^{\frac{1}{2}}f\|_{L_x^\infty L_p^2}\|\omega f\|_{L_{x,p}^\infty}\|
    \omega^{\frac{1}{2}}f\|_{L_{x,p}^1}\,ds\\
    &\le C
    \langle t\rangle^{\frac{1}{2}}\int_0^{t} | t-s|^{-\frac{1}{2}}\|\omega^\gamma f\|_{L_x^\infty L_p^2}\|\omega^\alpha f\|_{L_{x,p}^\infty}\|
    f\|_{L_{x,p }^1}\,ds\\
    &\le CC_1^3\langle t\rangle^{\frac{1}{2}}\int_0^{t}| t-s|^{-\frac{1}{2}} | s|^{-\frac{1}{2}}\varepsilon^{3}\,ds
    \le CC_1^3\langle t\rangle^{\frac{1}{2}}\varepsilon^{3},
\end{align*}
provided \(\gamma\le \tfrac{1}{2}\), and \(\alpha\le 1\). Consequently,
\[
I_1<\frac{5}{12} C_1\varepsilon
\]
whenever
\(C C_1 \varepsilon^3T^{\frac{1}{2}}\le \frac{5}{12}C_1\varepsilon\), i.e.
\(T\lesssim\varepsilon^{-4}\).
\end{proof}

\begin{remark}
Although the dispersive estimate yields an optimal decay rate of $\frac{1}{t}$ for the mapping 
\[
\mathcal{S}(t): L_x^1L_p^\infty(\omega^{-1}) \longrightarrow L_x^\infty L_p^1,
\]
such a decay is not sufficient to close the bootstrap argument (see the $L_x^\infty L_p^2$ estimate in the proof above). 
To overcome this issue, one must use the spaces $L_x^rL_p^\infty$ and $L_x^rL_p^1$, \(1<r<\infty\),
which respectively provide the decay rates $t^{-1/r}$ and $t^{-(1-1/r)}$. 
Balancing these two exponents gives the optimal choice $r=2$, 
which corresponds to the decay $t^{-1/2}$ used in our dispersive framework.
\end{remark}

\begin{corollary}[Conservation of mass and energy]\label{cor:conservation}
Under the same assumptions as in Theorem~\ref{thm:Longtime_dispersive},
the mild solution $f(t,p,x)$ constructed there satisfies the conservation laws for all $t\in[0,T]$:
\begin{align}
    \int_{\mathbb{R}}\!\int_{\mathbb{T}} f(t,p,x)\,dp\,dx
    &= \int_{\mathbb{R}}\!\int_{\mathbb{T}} f_0(p,x)\,dp\,dx, \label{eq:mass_cons}\\[4pt]
    \int_{\mathbb{R}}\!\int_{\mathbb{T}} \omega(p)\,f(t,p,x)\,dp\,dx
    &= \int_{\mathbb{R}}\!\int_{\mathbb{T}} \omega(p)\,f_0(p,x)\,dp\,dx. \label{eq:energy_cons}
\end{align}
\end{corollary}
\begin{proof}[Proof sketch]
The argument is identical to that of the \(L^1_{x,p}\)-bound in 
\textbf{Step~2} of the bootstrap scheme.
Using the Duhamel formulation and the fact that the collision operator
preserves the collision invariants \(1\) and \(\omega(p)\),
one obtains conservation of mass and energy.
\end{proof}
\section*{Acknowledgements}
The author would like to thank Pierre Germain and Michele Coti Zelati for their
supervision and continuous support.
The author also thanks Angeliki Menegaki and Zongguang Li for helpful discussions.

\bibliographystyle{plain} 
\bibliography{refs}  
\medskip
\noindent\textsc{Department of Mathematics,
Imperial College London,
London SW7 2AZ, United Kingdom}

\noindent\textit{Email address:} \texttt{h.xiang23@imperial.ac.uk}

\end{document}